\documentclass[11pt,twoside,reqno]{amsart}

\usepackage{amsmath,amsfonts,amsthm,amssymb}
\usepackage{graphicx,psfrag,overpic}

\usepackage{cite}
\usepackage{bm,enumerate}
\usepackage{cases}
\usepackage[all,cmtip,line]{xy}
\usepackage{array,CJK}
\usepackage{color}

\numberwithin{equation}{section}
\numberwithin{figure}{section}

\usepackage{graphics}
\usepackage{epsfig}
\newtheorem{claim}{\bf \t}[part]

\newtheorem{theorem}{Theorem}[section]

\newtheorem{lemma}{Lemma}[section]
\newtheorem{proposition}[theorem]{Proposition}

\def\ep{\epsilon}
\def\e{\varepsilon}

\def\a{\alpha}
\def\b{\beta}
\def\g{\gamma}

\def\l{\lambda}

\def\r{\rho}

\def\o{\omega}

\def\pt{\partial}
\def\iy{\infty}
\def\div{\mathrm{div}}

\def\S{\mathcal{S}}
\def\n{\bm{n}}
\def\tn{\bm{\widetilde{n}}}
\def\io{\bm{\iota}}
\def\tio{\widetilde{\bm{\iota}}}
\def\t{\bm{\tau}}
\def\tt{\widetilde{\bm{\tau}}}

\def\tr{\tilde{\rho}}

\def\tu{\widetilde{u}}
\def\tg{\widetilde{G}}
\def\tpsi{\widetilde{\psi}}
\def\tnabla{\widetilde{\nabla}}
\def\vp{\varphi}
\def\tw{\widetilde{W}}
\def\br{\bar{\rho}}
\def\sg{\sigma}
\def\ud{\dot{u}}
\def\vo{\vec{\omega}}

\def\w41{\|\omega\|_{W^{4,1}(0,\iy)}}
\def\lap{\triangle}
\def\D{\mathbf{D}}

\newcommand \R{\mathbb{R}}

\begin{document}

\title[]
{Global existence theorem for the three-dimensional isentropic compressible Navier-Stokes flow in the exterior of a rotating obstacle}
\author{Tuowei Chen}
\address{Tuowei Chen,
    School of Mathematical Sciences, Fudan University,
	Shanghai 200433, P. R. China;}
\email{twchen16@fudan.edu.cn}

\author{Yongqian Zhang}
\address{Yongqian Zhang, School of Mathematical Sciences, Fudan University, Shanghai 200433, P. R. China}
\email{yongqianz@fudan.edu.cn}
\date{\today}

\begin{abstract}
\,\,In this paper, we consider the initial-boundary value problem of three-dimensional isentropic compressible Navier-Stokes equations with rotating effect terms in an exterior domain with Navier-slip boundary condition and with far-field vacuum. This problem is related to the motion of the compressible viscous flow past a rotating obstacle. We establish the global existence and uniqueness of classical solutions, provided that the initial mass is small. The initial data and the angular velocity of the obstacle are allowed to have large oscillations.
\end{abstract}
\keywords{compressible Navier-Stokes equations, global classical solution, exterior domain, rotating obstacle, vacuum}
\subjclass[2010]{
	35Q30; 
	76N15; 
    76U05; 
}
\maketitle

\section{Introduction}
In this paper, we consider the following three-dimensional isentropic compressible Navier-Stokes equations with rotating effect terms:
\begin{equation}
\left.\begin{cases}
\r_t+(U-\vec{\omega}\times x)\cdot\nabla\r+\r\div U=0,\,\,\,  x\in \Omega,\,t>0 \\
\r U_{t}+\r (U-\vec{\omega}\times x)\cdot\nabla U+\r\vo\times U-2\mu \div \D(U)-\lambda\nabla\div U+\nabla P(\r)=0.
\end{cases}\label{eq}
\right.
\end{equation}
Here, $\r\geq0$, $U=(U^1,U^2,U^3)$, and $P(\r)=a\r^\g(a>0,\g>1)$ are the fluid density, velocity and pressure, respectively. $\D(u)$ is the strain tensor,
\begin{align}
\D(U)=\frac{1}{2}\big(\nabla U+(\nabla U)^T\big),
\end{align}
where $\cdot^T$ stands for the transposition. $\mu$ and $\lambda$ are the constant viscosity coefficients, and in this paper, it will be always assumed that
\begin{align}
\mu > 0 \,\,\,\,\,\,\mathrm{and}\,\,\,\,\,\,\,2\mu+3\lambda > 0.
\end{align}
$\vo=\vec{\omega}(t)=\big(0,0,\omega(t)\big)$ and $\omega(t)$ is a given smooth function on $[0,\iy)$.

$\Omega$ is an exterior domain, that is
\begin{align}
\Omega=\R^3\setminus \S,\label{Omega}
\end{align}
where $\S$ is a compact and simply connected subset in $\R^3$ such that $(0,0,0)\in\S$. The boundary $\pt\Omega$ is a smooth hypersurface.

We look for global classical solutions to \eqref{eq}, with initial data
\begin{align}
\r(0,x)=\r_0(x),\,\,U(0,x)=U_0(x),\,\,\,\,\,\,\,x\in\Omega,
\end{align}
and with the boundary conditions
\begin{align}
\r(x,t)\rightarrow 0,\,\,\,\,U(x,t)\rightarrow 0,\,\,\,\mathrm{as}\,\,\,|x|\rightarrow \iy,\,(x,t)\in \Omega\times(0,\iy),\label{far}
\end{align}
and
\begin{equation}
\left.
\begin{cases}
(U-\vo\times x)\cdot \n=0,\,\,\,\,x\in\pt\Omega,t>0,\\
(\D(U)\n)_{\tau}+\a\big(U-(\vo\times x)\big)_{\tau}=0,\,\,\,\,x\in\pt\Omega,t>0.
\end{cases}\label{nab1}
\right.
\end{equation}
Here, $\n=(\n^1,\n^2,\n^3)$ stands for the outward unit normal to $\Omega$, and for any vector field $v$ on $\pt\Omega$, $v_{\tau}$ stands for the tangential part of $v$: $v_{\tau}=v-(v\cdot\n)$. In addition, $\a$ is the friction coefficient attached to the Navier-slip boundary condition \eqref{nab1} which was introduced by Navier \cite{Navier}.

We remark that our problem is concerning the motions of the compressible flow past a rotating rigid obstacle with angular velocity $\vo$, and it can be deduced by using a coordinate system attached to the rotating obstacle as in \cite{H1999,2014,CZ1,CZ2}.

There have been many works on incompressible flows past a rotating obstacle. Hishida \cite{H1998,H1999} first constructed a local mild solution within the framework of $L^2$, and later on, Geissert, Heck and Hieber \cite{G2006} extended this result to the $L^p$-case. For recent progress on global existence, we refer readers to \cite{H2009,H2014,H2018,H2019,T2020} and the references therein. For compressible flows past a rotating obstacle, Kra\v{c}mar, S., Ne\v{c}asov\'{a}, \v{S}. and Novotn\'{y}, A. \cite{2014} considered the motion of compressible viscous fluids around a rotating rigid obstacle with positive density and nonzero velocity at infinity, and they proved the global existence of weak solutions. The regularities and uniqueness of such weak solutions still remain open. Farwig and Pokorn\'{y} \cite{Farwig} considered a linearized model for compressible flows past a rotating obstacle in $\R^3$ with positive density at infinity, and they proved the existence of solutions in $L^q$-spaces. For the case with initial density having a compact support, the present authors \cite{CZ1} showed the local existence of a unique strong solution. For the case that the flow is restricted in a bounded domain, the present authors \cite{CZ2} obtained the global existence of a unique strong solution, provided that the initial data are close to some approximate solution.

For the multi-dimensional Navier-Stokes equations without rotating effect terms, Matsumura and Nishida \cite{Matsumura1,Matsumura2} first obtained the global existence of strong solutions with initial data close to a nonvacuum equilibrium in Sobolev space $H^3$. When an initial vacuum is allowed, the local existence and uniqueness of strong solutions and classical solutions are known in \cite{cho2004,cho2006,choe2003,Huang2020}. For the global existence of weak solutions, a breakthrough is due to Lions \cite{Lions}, in which he proved the global existence of weak solutions with finite energy, when the adiabatic exponent $\g$ is suitably large. Later, this result was improved by Feireisl \cite{Feireisl}. In addition, Hoff \cite{Hoff2005} obtained a new type of global weak solutions with small energy, which have more regularity information compared with the ones in \cite{Lions,Feireisl}. Recently, Huang, Li and Xin \cite{HLX2012} established the global existence and uniqueness of classical solutions to the Cauchy problem of \eqref{eq} with initial data that are of small energy but possibly large oscillations with constant state as far field, which could be vacuum or non-vacuum. There are a lot of works that extend the result in \cite{HLX2012}, for which we refer readers to \cite{Duan2012,YZ2017,SZZ2018,LX2018,LX2019}. However, to the best of our knowledge, there is little literature on the global existence of classical solutions to the exterior problem with initial vacuum and with initial data which have possibly large oscillations.

In this paper, we intend to find a global classical solution to \eqref{eq}--\eqref{nab1} with possibly large angular velocity $\o(t)$. Before stating the main result, we list the following simplified notations and settings used in this paper.

$\bullet$  (Adopted from \cite{DV2002}). We say that $\S$ is axisymmetric if it admits an axis of symmetry, which means that it is preserved by a rotation of arbitrary
angle around this axis.

$\bullet$  $B_R(x)=\{y\in\R^3\mid |y-x|\leq R\}$,\,\,\,$B_R=B_R(0)$,\,\,\,$B_R^c=\R^3\setminus B_R$.

$\bullet$  $\mathrm{supp}f=\overline{\{x\in\R^2\mid f(x)\neq0\}}$ denotes the support of $f$.

$\bullet$  $\pt_if\triangleq\frac{\pt f}{\pt x_i}$,\,\,\,$\pt_{v}\triangleq v\cdot\nabla$,\,\,\,$u^v\triangleq u\cdot v$ for any vector fields $u,v$.

$\bullet$ The standard homogeneous and inhomogeneous Sobolev spaces are defined as follows:
\begin{align}
&L^r=L^r(\Omega),\,\,\,\,\,\,D^{k,r}=\{u\in L^1_{loc}(\Omega):\, \|u\|_{D^{k,r}}<\iy\},\,\,\,\,\,\,\|u\|_{D^{k,r}}\triangleq\|\nabla^k u\|_{L^r},\notag\\
&W^{k,r}=L^r\cap D^{k,r},\,\,\,\,D^k=D^{k,2},\,\,\,\,\,\,H^k=W^{k,2},\,\,\,\,\,\,H^1_0=\{f\in H^1\mid f=0 \,\,\mathrm{on}\,\,\pt\Omega\}.\notag
\end{align}

$\bullet$ The notation $|\cdot|_{H^m(\pt\Omega)}$ will be used for the standard Sobolev norm of functions defined on $\pt\Omega$. Note that this norm involves derivatives only along the boundary.

$\bullet$ BMO$(\Omega)$ stands for the John-Nirenberg space of bounded mean oscillation whose norm is defined by
\begin{align}
\|f\|_{BMO}\triangleq \|f\|_{L^2}+[f]_{BMO},\notag
\end{align}
with
\begin{align}
[f]_{BMO}=\sup_{x\in\Omega,r>0}\frac{1}{|\Omega\cap B_r(x)|}\int_{\Omega\cap B_r(x)}|f(y)-f_{\Omega\cap B_r(x)}|dy,\notag
\end{align}
and
\begin{align}
f_{\Omega\cap B_r(x)}=\frac{1}{|\Omega\cap B_r(x)|}\int_{\Omega\cap B_r(x)}f(y)dy.\notag
\end{align}

In addition, without loss of generality, we assume that $\S\subset B_1$.

We now state the main results in this paper as follows:
\begin{theorem}\label{main}
Let $\o\in W^{4,1}(0,\iy)$. Suppose that $\S$ and $\a=\a(x)$ satisfy either of the following two cases:\\
(1)\,$\S$ is non-axisymmetric and $0\leq\a\in C^{\iy}(\pt\Omega)$;\\
(2)\,$\S$ is axisymmetric, $0\leq\a\in C^{\iy}(\pt\Omega)$, and there exists $x_0\in\pt\Omega$ such that $\a(x_0)>0$.

In addition, for given numbers $M>0$ and $\br>1$, suppose that the initial data $(\r_0, U_0)$ satisfy
\begin{align}
&U_0\in D^1\cap D^2,\,\,\,\r_0\in L^1\cap H^2\cap W^{2,6},\,\,\,P(\r_0)\in H^2\cap W^{2,6},\,\,\,|x|\r_0^{\frac{1}{2}}\in L^\iy,\label{data1}\\
&0\leq\inf\r_0\leq \sup\r_0\leq \br,\,\,\,\,\|\nabla U_0\|_{L^2}+\|U_0\|_{D^2}\leq M,\label{data2}
\end{align}
and the compatibility conditions
\begin{equation}
\left.
\begin{cases}
(U_0-\vo(0)\times x)\cdot \n=0,\,\,\,\,x\in\pt\Omega,\\
(\D(U_0)\n)_{\tau}+\a\big(U_0-(\vo(0)\times x)\big)_{\tau}=0,\,\,\,\,x\in\pt\Omega,\\
-2\mu\div\D(U_0)-\lambda\nabla\div U_0+\nabla P(\r_0)=\r_0^{\frac{1}{2}}g,\,\,\,\,x\in\Omega
\end{cases}
\right.\label{data3}
\end{equation}
for some $g\in L^2$. Then, there exists a positive constant $\e$ depending at most on $\mu$, $\l$, $a$, $\g$, $\br$, $\Omega$, $\w41$ and $M$, such that if
\begin{align}
\int_{\Omega}\r_0dx\leq \e,\label{inimass}
\end{align}
then the initial-boundary value problem \eqref{eq}--\eqref{nab1} has a unique global classical solution $(\r,U)$ satisfying
\begin{align}
0\leq \r(x,t)\leq 2\br, \,\,\,\,\,x\in\Omega,\,t\geq0,\label{upperden}
\end{align}
and
\begin{equation}
\left.
\begin{cases}
\big(\r,P(\r)\big)\in C\big([0,T];H^2\cap W^{2,6}\big),\\
\nabla U\in L^\iy([0,T];H^1)\cap L^\iy_{\mathrm{loc}}(0,T;H^2\cap W^{2,6})\cap C([0,T];H^1),\\
U_t\in L^\iy_{\mathrm{loc}}(0,T;D^1\cap D^2)\cap H^1_{\mathrm{loc}}(0,T;D^1),\\
\sqrt{\r}U_t\in L^\iy(0,T;L^2)
\end{cases}\label{class}
\right.
\end{equation}
for any $0<T<\iy$.

Moreover, if the $\r_0$ has a compact support such that $\mathrm{supp}\r_0\subset \overline{B_{R_0}}$ for $R_0>1$, then these exists a positive constant $C$ such that $\rho(t)$ has a compact support for $t \in (0,T)$ and satisfies
\begin{align}
sup\{ |x|: \r(x,t)>0\}\leq R_{0}+Ct.\label{compact}
\end{align}
\end{theorem}

We remark that compared to \cite{HLX2012}, an essential difficulty in our problem comes from the non-homogeneous boundary condition \eqref{nab1}, since the angular velocity $\o$ is allowed to be large. If one reduces \eqref{nab1} to a homogeneous one by letting $U$ minus some auxiliary function, then some extra (possibly large) source terms will arise in the equation $\eqref{eq}_2$ such that it is not easy to derive an energy estimate as in \cite{HLX2012} even if the initial energy is small. In order to overcome this difficulty, we introduce an auxiliary function with small second derivative (see \eqref{b}), and then obtain the energy estimate in the Lemma \ref{Energy} under condition that the initial mass is small. Moreover, compared to the case in a half-space \cite{Duan2012} and the case in a bounded cuboid domain \cite{YZ2017}, some new difficulties appear in the case in an exterior domain with non-flat boundary, mainly due to the fact that $\n$ is no longer a constant vector field such that the classical technique of localization and flattening the boundary cannot be applied to obtain the estimation on the vorticity $\nabla\times u$. To overcome such difficulties, we use local frames near the boundary to establish the $L^p$-norm estimation on $\nabla\times u$ in the Lemma \ref{FW}. In addition, we need modify the material derivative by adding an auxiliary function related to the second second fundamental form (see \eqref{uddef}) such that normal derivative of $u$ will not appear in the integral on the non-flat boundary\footnote{This paper had been submitted. After that, recently we knew that there had been some recent progress on global classical solutions to the compressible Navier-Stokes equations with slip boundary conditions in both of 3D bounded domains and 3D exterior domains\cite{CL2021,CLL2021}, where the estimates on the boundary were obtained by different methods.}. Moreover, some difficulties comes from the rotating effect terms such as $(\vo\times x)\cdot\nabla u$ which will blow up at infinity. To deal with such terms, we use the method of integration by parts abundantly and establish the estimation on the $L^\iy$-norm of $|x|\r$ under assumption that the initial density decays to the vacuum at the far field.

The remaining of this paper is organized as follows: Section 2 is devoted to introducing the local existence results and some elementary inequalities which will be used in later analysis; In Section 3, we first derive the necessary \textit{a priori} estimates on the smooth solutions of a reduced problem, and then prove the main theorem.

\section{Preliminaries}
In this section, we recall some known facts and introduce some useful inequalities that will be used frequently later.

We start with the local existence result as follows:
\begin{proposition}\label{Local}
Let $\o\in W^{4,1}(0,\iy)$. Suppose that $\S$ and $\a$ satisfy either of the two cases stated in the Theorem \ref{main}. Suppose that the initial data $(\r_0, U_0)$ satisfy \eqref{data1}--\eqref{data3}. In addition, suppose that the initial density $\r_0$ has a compact support. Then, there exist a small time $T_*>0$ and a unique classical solution $(\r,U)$ to the initial-boundary value problem \eqref{eq} on $\Omega\times(0,T_*]$ such that
\begin{equation}
\left.
\begin{cases}
(\r,P(\r))\in C\big([0,T_*];H^2\cap W^{2,6}\big),\\
U\in C([0,T_*];D^1\cap D^2),\,\,\nabla U\in L^2(0,T_*;H^2),\\
\sqrt{t}\nabla U\in L^\iy(0,T_*;H^2),\,\,t\nabla U\in L^\iy(0,T_*;W^{2,6}),\\
\sqrt{\r}U_t\in L^\iy(0,T_*;L^2),\,\,\sqrt{t}U_t\in L^\iy(0,T_*;D^1)\cap L^2(0,T_*;D^2),\\
tU_t\in L^\iy(0,T_*;D^2),\,\,tU_{tt}\in L^2(0,T_*;D^1),\\
t\sqrt{\r}U_t\in L^\iy(0,T_*;L^2),\,\,\sqrt{t}\sqrt{\r}U_{tt}\in L^2(0,T_*;L^2).
\end{cases}\label{localclass}
\right.
\end{equation}
\end{proposition}

\begin{proof}
The local existence result can be proved by using the combination of the methods in \cite{CZ1,cho2006,Huang2020}, and we thus omit the details.
\end{proof}

Next, we show that Sobolev's inequality holds for vector fields that are tangential to the boundary $\pt\Omega$.
\begin{lemma}\label{sob}
Let $u\in D^1(\Omega)$ such that $u\cdot\n=0$ on $\pt\Omega$. Then, it holds that
\begin{align}
\|u\|_{L^6}\leq C\|\nabla u\|_{L^2}\label{sobolev}
\end{align}
for a positive constant $C$ depending only on $\Omega$.
\end{lemma}
\begin{proof}
Applying the Sobolev imbedding theorem (see \cite{Evans}) on the bounded domain $\Omega_5$, we can use proof by contradiction to show that
\begin{align}
\|u\|_{H^1(\Omega_5)}\leq C\|\nabla u\|_{L^2(\Omega_5)}\leq C\|\nabla u\|_{L^2}.
\end{align}

Then, it follows from the trace theorem (see \cite{Evans}) that there exists $\tu\in D^1(\R^3)$ such that $\tu|_{\Omega}=u$ and
\begin{align}
\|\tu\|_{H^1(\S)}\leq C|u|_{H^{\frac{1}{2}}(\pt\Omega)}\leq C\|u\|_{H^1(\Omega_5)}\leq  C\|\nabla u\|_{L^2},
\end{align}
which, together with Sobolev's inequality, implies that
\begin{align}
\|u\|_{L^6}\leq\|\tu\|_{L^6(\R^3)}\leq C\|\nabla \tu\|_{L^2(\R^3)}\leq C(\|\nabla u\|_{L^2}+\|\nabla \tu\|_{L^2(\S)})\leq  C\|\nabla u\|_{L^2}.
\end{align}

The proof is completed.
\end{proof}

We now derive the following Korn's type inequality in the exterior domain $\Omega$.
\begin{lemma} \label{korn}
Let $u\in D^1(\Omega)$ such that $u\cdot\n=0$ on $\pt\Omega$. Suppose that $\S=\R^3\setminus\Omega$ and $\a$ satisfy either of the following two cases:\\
(1)\,$\S$ is non-axisymmetric and $0\leq\a\in C^{\iy}(\pt\Omega)$;\\
(2)\,$\S$ is axisymmetric, $0\leq\a\in C^{\iy}(\pt\Omega)$, and there exists $x_0\in\pt\Omega$ such that $\a(x_0)>0$.

Then, there exist a positive constant $C$ depending at most on $\Omega$ and $\a$ such that
\begin{align}
\|\nabla u\|_{L^2}\leq C\big(\|\D(u)\|_{L^2}+|\sqrt{\a}u|_{L^2(\pt\Omega)}\big).\label{korntype}
\end{align}
\end{lemma}
\begin{proof}
For the case (1), we can use the similar method as in the proof of the Theorem 1 in \cite{Korn2019} to obtain that $\Omega_5$ cannot support rigid motions that are tangential to the boundary $\pt\Omega$ (for more details see \cite{Korn2019} and reference therein). Thus, we can use use proof by contradiction as in the proof of the Theorem 2 in \cite{Korn2019} to obtain that
\begin{align}
\| u\|_{H^1(\Omega_{5})}\leq C\|\D(u)\|_{L^2(\Omega_5)}.
\end{align}
Therefore, by the trace theorem, there exists $\tu\in D^1(\R^3)$ such that $\tu|_{\Omega_5}=u$ and
\begin{align}
\|\tu\|_{H^1(\S)}\leq C|u|_{H^{\frac{1}{2}}(\pt\Omega)}\leq C\|u\|_{H^1(\Omega_5)}\leq C\|\D(u)\|_{L^2(\Omega_5)}\leq C\|\D(u)\|_{L^2},\label{k1}
\end{align}
which, together with the Fourier transform, gives that
\begin{align}
\|\nabla u\|_{L^2}\leq\|\nabla\tu\|_{L^2(\R^3)}\leq 2\| \D(\tu)\|_{L^2(\R^3)}
\leq C(\|\D(u)\|_{L^2}+\|\nabla \tu\|_{L^2(\S)})
\leq C\|\D(u)\|_{L^2}.\label{k2}
\end{align}

For the case (2), similarly, we can use proof by contradiction to obtain that
\begin{align}
\| u\|_{H^1(\Omega_{5})}\leq C\big(\|\D(u)\|_{L^2(\Omega_5)}+|\sqrt{\a}u|_{L^2(\pt\Omega)}\big).
\end{align}
Then, we can get \eqref{korntype} in a similar manner.

The proof is completed.
\end{proof}

The following is the Gagliardo-Nirenberg inequality in $\Omega$.
\begin{lemma}\label{gn}
For $p\in[2,6]$, $q\in(1,\iy)$ and $r\in(3,\iy)$, there exists some generic constant $C>0$ which may depend on $q,r$ such that for $f\in H^1(\Omega)$ and $g\in L^q(\Omega)\cap W^{1,r}(\Omega)$, we have
\begin{align}
&\|f\|_{L^p}\leq C\|f\|^{\frac{6-p}{2}}_{L^2}\|f\|^{\frac{3p-6}{2}}_{H^1},\\
&\|g\|_{C(\overline{\Omega})}\leq C\|g\|^{\frac{q(r-3)}{3r+q(r-3)}}_{L^q}\|g\|^{\frac{3r}{3r+q(r-3)}}_{W^{1,r}}.
\end{align}
\end{lemma}
\begin{proof}
The Lemma \ref{gn} is a direct result form the Sobolev extension theorem (see \cite{Evans}) and the well-known Gagliardo-Nirenberg inequality for the whole space.
\end{proof}

Next, let $u$ solve the following boundary value problem for the Lam\'{e} operator $L$:
\begin{align}
Lu\triangleq2\mu\div\D(u)+\l\nabla\div u=F\,\,\,\,\mathrm{in}\,\,\,\Omega,\label{eq1}
\end{align}
\begin{equation}
\left.
\begin{cases}
u\cdot \n=0\,\,\,\,\,\,\,\mathrm{on}\,\,\,\pt\Omega,\\
(\D(u)\n)_{\tau}+\a u_{\tau}=0\,\,\,\,\,\,\,\mathrm{on}\,\,\,\pt\Omega.
\end{cases}\label{nab2}
\right.
\end{equation}

We now derive some estimates for the above elliptic system in the exterior domain $\Omega$.
\begin{lemma}\label{exterior}
Let $p\in[2,6]$, and let $u$ be a smooth solution of \eqref{eq1}--\eqref{nab2}. Then, there exists a positive constant $C$ depending only on $\mu$, $\l$, $p$ and $\Omega$ such that the following estimates hold:
(1) if $F\in L^p(\Omega)$, then
\begin{align}
\|u\|_{D^{2,p}}\leq C(\|F\|_{L^p}+\|\nabla u\|_{L^p});\label{ell1}
\end{align}
(2) if $F=\div f$ with $f=(f^{ij})_{3\times 3}$, $f^{ij}\in L^p(\Omega)$, then
\begin{align}
\|\nabla u\|_{L^p}\leq C(\|f\|_{L^p}+\|u\|_{L^6});\label{ell2}
\end{align}
(3) if $F=\div f$ with $f=(f^{ij})_{3\times 3}$, $f^{ij}\in L^\iy(\Omega)\cap L^2(\Omega)$, then $\nabla u\in BMO(\Omega)$ and
\begin{align}
\|\nabla u\|_{BMO(\Omega)}\leq C(\|f\|_{L^\iy}+\|f\|_{L^2}+\|u\|_{L^\iy}).\label{ell3}
\end{align}
\end{lemma}
\begin{proof}
(1) Based on the standard $L^p$-estimate for the strongly elliptic systems (see \cite{ADN1959,ADN1964}), the deducing of \eqref{ell1} is on the same lines as in the Section 5 of \cite{cho2004}.

(2) Choosing a cut-off function $\varphi\in C^\iy_c(\R^3)$ such that $\vp=1$ in $B_5$ and $\vp=0$ on $B^c_{10}$, we define
\begin{align}
v=\vp u\,\,\,\,\,\,\,\,\mathrm{and}\,\,\,\,\,\,\,\,w=(1-\vp u)\equiv \psi u.
\end{align}
Then, for $j=1,2,3$, we have
\begin{align}
Lv^j=&\vp Lu^j+2\mu\sum^3_{i=1}\vp_{x_i}u^j_{x_i}+\mu\triangle\vp u^j+(\mu+\l)\vp_{x_j}\div u+(\mu+\l)\big(\sum^3_{i=1}\vp_{x_i}u^i\big)_{x_j}\notag\\
=&\sum^3_{i=1}\big[(\vp f^{ij})_{x_i}-\vp_{x_i}f^{ij}\big]+2\mu\sum^3_{i=1}(\vp_{x_i}u^j)_{x_i}-\mu\triangle\vp u^j\notag\\
&+(\mu+\l)\sum^3_{i=1}\big[(\vp_{x_j}u^i)_{x_i}
-\vp_{x_ix_j}u^i\big]+(\mu+\l)\sum^3_{i=1}\big(\delta_{ij}\sum^3_{k=1}\vp_{x_k}u^k\big)_{x_i}\notag\\
\equiv& \sum^3_{i=1}\big(H^{ij}\big)_{x_i}+h^j,
\end{align}
where $H^{ij}=\vp f^{ij}+2\mu\vp_{x_i}u^j+(\mu+\l)\vp_{x_j}u^i+(\mu+\l)\delta_{ij}\sum^3_{k=1}\vp_{x_k}u^k$, $h^j=-\vp_{x_i}f^{ij}-\mu\triangle\vp u^j-(\mu+\l)\vp_{x_ix_j}u^i$, and
\begin{numcases}{\delta_{ij}=}
1,\,\,\,\,\,\,\,\mathrm{if}\,\,i=j;\notag\\
0,\,\,\,\,\,\,\,\mathrm{if}\,\,i\neq j.\notag
\end{numcases}
Thus, if we let $H=(H^{ij})_{3\times3}$ and $h=(h^1,h^2,h^3)$, then we have
\begin{equation}
\left.\begin{cases}
Lv-h=\div H\,\,\,\,\mathrm{in}\,\,\Omega_{10},\\
v\,\,\,\,\, \mathrm{satisfies}\,\,\eqref{nab2}\,\,\,\,\mathrm{on}\,\,\partial\Omega,\\
v=0\,\,\,\,\,\mathrm{on}\,\,\,\,\partial B_{10}.\label{eqv}
\end{cases}
\right.
\end{equation}

Next, let $v_h\in W^{1,p}(\Omega_{10})$ be the unique weak solution to the following elliptic system:
\begin{equation}
\left.\begin{cases}
Lv_h=h\,\,\,\,\mathrm{in}\,\,\Omega_{10},\\
v_h\,\,\,\,\, \mathrm{satisfies}\,\,\eqref{nab2}\,\,\,\,\mathrm{on}\,\,\partial\Omega,\\
v=0\,\,\,\,\,\mathrm{on}\,\,\,\,\partial B_{10}.
\end{cases}
\right.
\end{equation}
The standard $L^p$-estimate for strongly elliptic systems \cite{ADN1959,ADN1964}, together with H\"{o}lder's inequality, gives that
\begin{align}
\|v_h\|_{W^{2,p}(\Omega_{10})}\leq C\|h\|_{L^p(\Omega_{10})}\leq C(\|f\|_{L^p}+\|u\|_{L^6}).
\end{align}
Therefore, we have
\begin{equation}
\left.\begin{cases}
L(v-v_h)=\div H\,\,\,\,\mathrm{in}\,\,\Omega_{10},\\
v-v_h\,\,\,\,\, \mathrm{satisfies}\,\,\eqref{nab2}\,\,\,\,\mathrm{on}\,\,\partial\Omega,\\
v-v_h=0\,\,\,\,\,\mathrm{on}\,\,\,\,\partial B_{10}.
\end{cases}
\right.
\end{equation}
Again, it follows from the classical theory for strongly elliptic systems \cite{ADN1959,ADN1964} that
\begin{align}
\|\nabla (v-v_h)\|_{L^p(\Omega_{10})}\leq C\|H\|_{L^p(\Omega_{10})}\leq C(\|f\|_{L^p}+\|u\|_{L^6}).
\end{align}
Consequently, we have
\begin{align}
\|\nabla v\|_{L^p(\Omega_{10})}\leq \|\nabla (v-v_h)\|_{L^p(\Omega_{10})}+\|\nabla v_h\|_{L^p(\Omega_{10})}\leq C(\|f\|_{L^p}+\|u\|_{L^6}).\label{estv}
\end{align}

Similarly, $w$ satisfies
\begin{align}
Lw-g=\div G\,\,\,\,\,\,\,\,\,\,\mathrm{in}\,\,\,\R^3,\label{eqw}
\end{align}
where $G^{ij}=\psi f^{ij}+2\mu\psi_{x_i}u^j+(\mu+\l)\psi_{x_j}u^i+(\mu+\l)\delta_{ij}\sum^3_{k=1}\psi_{x_k}u^k$, $g^j=-\sum^3_{i=1}\psi_{x_i}f^{ij}-\mu\triangle\psi u^j-(\mu+\l)\sum^3_{i=1}\psi_{x_ix_j}u^i$. Here, we have extended $w$, $g$ and $G$ to all of $\R^3$ by setting them equal to zero on $\S$.
Consider the equation
\begin{align}
Lw_g=g\,\,\,\,\,\,\,\,\,\,\,\,\mathrm{in}\,\,\R^3.\label{eqwg}
\end{align}
There is an explicit representation formula for the solution to \eqref{eqwg} as follows (see the Chapter 10 in \cite{Mclean2000}):
\begin{align}
w_g(x)=\int_{\R^2}\Gamma(x-y)g(y)dy\,\,\,\,\,\,\,\,\forall\,x\in\, \R^3,
\end{align}
where $\Gamma(x-y)=\frac{1}{8\pi\mu(2\mu+\l)}\big((3\mu+\l)|x-y|\mathrm{I}_3+(\mu+\l)\frac{(x-y)(x-y)^T}{|x-y|^3}\big)$, $\mathrm{I}_3$ is the $3\times3$ unit matrix. Therefore, we have
\begin{align}
L(w-w_g)=\div G \,\,\,\,\,\,\,\,\,\,\,\,\mathrm{in}\,\,\R^3.
\end{align}
We can obtain from the classical $L^p(\R^2)$-estimate for singular integral operators \cite{Stein1993} that
\begin{align}
\|\nabla (w-w_g)\|_{L^p(\R^3)}\leq C\|G\|_{L^p(\R^3)}\leq C(\|f\|_{L^p}+\|u\|_{L^6}),\label{dual0}
\end{align}
where we have used the fact that all the derivatives of $\psi$ have compact supports subset of $B_{10}\setminus B_5$.

Thus, it suffices to estimate $\|\nabla w_g\|_{L^p(\R^3)}$. For $p\in(2,6]$, let $q$ satisfy that $\frac{1}{q}+\frac{1}{p}=1$ and let $r>10$ be fixed. For $\phi\in L^q(B_r)$, we extend $\phi$ to $\R^3$ by setting it equal to zero outside of $B_r$. Then, for $w^j_g(x)=\int_{\R^3}\sum^3_{k=1}\Gamma^{jk}(x-y)g^k(y)dy$, we have
\begin{align}
\|\partial_{x_i} w^j_g\|_{L^p(B_r)}=&\sup_{\|\phi\|_{L^q(B_r)}\leq 1}\big|\int_{\R^3}\partial_{x_i} w^j_g(x)\phi(x) dx\big|\notag\\
=&\sup_{\|\phi\|_{L^q(B_r)}\leq 1}\big|\int_{\R^3}\sum^3_{k=1}g^k(y)\Phi_k(y)dy\big|\label{dual1}
\end{align}
for $1\leq i,j\leq 3$, where $\Phi_k(y)=\int_{\R^3}\Gamma^{jk}(x-y)\phi^k(x)dx$.

Next, since $\frac{6}{5}\leq q<2$, we denote by $D^{1,q}_0(\R^3)$ the completion of $C^\iy_0(\R^3)$ in the norm $\|f\|_{D^{1,q}_0(\R^3)}=\|\nabla f\|_{L^q(\R^3)}$ as in \cite{Galdi}, and denote by $D^{-1,p}_0(\R^3)$ the completion of $C^\iy_0(\R^3)$ in the norm
\begin{align}
\|f\|_{D^{-1,p}_0(\R^3)}=\sup_{z\in D^{1,q}_0(\R^3),\,\|z\|_{D^{1,q}_0(\R^3)}=1}\big|\int_{\R^3}fzdx\big|.\label{duiou}
\end{align}
It is shown in the Chapter 2 of \cite{Galdi} that $D^{-1,p}_0(\R^3)$ and $\big(D^{1,q}_0(\R^3)\big)^\prime$ are isomorphic, where $\big(D^{1,q}_0(\R^3)\big)^\prime$ is the normed dual space of $D^{1,q}_0(\R^3)$. We claim that $\Phi_k\in D^{1,q}_0(\R^3)$. Indeed, it follows from the classical Calderon-Zygmund Theorem \cite{Stein1993} that
\begin{align}
\|\nabla \Phi_k\|_{L^q(\R^3)}\leq C\|\phi\|_{L^q(B_r)},\label{dual2}
\end{align}
where $C$ is a positive constant independent of $r$. Moreover, since $\phi$ has compact support, we have $\Phi_{k}(y)\rightarrow 0$ as $|y|\rightarrow \iy$. Hence, $\Phi_k\in D^{1,q}_0(\R^3)$. Then, \eqref{dual1} and \eqref{dual2} lead to
\begin{align}
\|\partial_{x_i} w^j_g\|_{L^p(B_r)}\leq C\|g\|_{D^{-1,p}_0(\R^3)}.\label{dual3}
\end{align}
Noting that $\mathrm{supp}g\subset B_{10}\setminus B_5$, we obtain, after using \eqref{duiou} and applying H\"{o}lder's inequality and Sobolev's inequality, that
\begin{align}
\|g\|_{D^{-1,p}_0(\R^3)}=&\sup_{z\in D^{1,q}_0(\R^3),\,\|z\|_{D^{1,q}_0(\R^3)}=1}\big|\int_{\R^3}gzdx\big|.\notag\\
\leq&\sup_{z\in D^{1,q}_0(\R^3),\,\|z\|_{D^{1,q}_0(\R^3)}=1}\|g\|_{L^{\frac{3p}{3+p}}(\R^3)}\|z\|_{L^{\frac{3q}{3-q}}(\R^3)}\notag\\
\leq&C\|g\|_{L^p(B_{10}\setminus B_5)}\leq C(\|f\|_{L^p}+\|u\|_{L^6}).\label{dual4}
\end{align}
Thus, letting $r\rightarrow\iy$, we obtain from  \eqref{dual3} and \eqref{dual4} that
\begin{align}
\|\nabla w_g\|_{L^p}\leq C(\|f\|_{L^p}+\|u\|_{L^6})\,\,\,\,\,\,\,\,p\in (2,6]. \label{p26}
\end{align}

For the case that $p=2$, multiplying \eqref{eqwg} by $w_g$, we obtain, after integrating the resulting equality over $\R^3$, together with using the method of integration by parts, that
\begin{align}
\|\nabla w_g\|_{L^2(\R^3)}\leq C\|g\|_{L^2(\R^3)}\leq C(\|f\|_{L^2(\R^3)}+\|u\|_{L^6}).
\end{align}
This, together with \eqref{p26}, \eqref{dual0} and \eqref{estv}, leads to \eqref{ell2}.

(3) In the proof we use the same notations as in the part (2). Similarly, it suffices to estimate $\|\nabla v\|_{BMO(\Omega_{10})}$ and $\|\nabla w\|_{BMO(\R^3)}$. First, noting that both of $h$ and $g$ have compact supports subset of $B_{10}\setminus B_5$, we extend $h$ and $g$ to $\R^3$ by setting them equal to zero outside of their supports, respectively. Then, we define a $3\times3$ matrix-value function $M_h=(M^{ij}_h)_{3\times3}$ as follows:
\begin{equation}
M^{ij}_h(x_1,x_2,x_3)=\left.\begin{cases}
\int_{-\iy}^{x_1}h^1(s,x_2,x_3)ds\,\,\,\,\,\,\mathrm{if}\,\,i=j=1,\\
\int_{-\iy}^{x_2}h^2(x_1,s,x_3)ds\,\,\,\,\,\,\mathrm{if}\,\,i=j=2,\\
\int_{-\iy}^{x_3}h^3(x_1,x_2,s)ds\,\,\,\,\,\,\mathrm{if}\,\,i=j=3,\\
0\,\,\,\,\,\,\,\,\,\,\,\,\mathrm{if}\,\,i\neq j.
\end{cases}\label{sameway}
\right.
\end{equation}
Therefore, it follows from \eqref{eqv} that
\begin{equation}
\left.\begin{cases}
Lv=\div (H+M_h)\,\,\,\,\mathrm{in}\,\,\Omega\cap B_{10},\\
v\,\,\,\,\, \mathrm{satisfies}\,\,\eqref{nab2}\,\,\,\,\mathrm{on}\,\,\partial\Omega,\\
v=0\,\,\,\,\,\mathrm{on}\,\,\,\,\partial B_{10}.
\end{cases}\label{bddsystem}
\right.
\end{equation}
Noting that \eqref{bddsystem} is a strongly elliptic system in a bounded domain, we can apply the similar method as in \cite{BMO1992} to obtain that
\begin{align}
\|\nabla v\|_{BMO(\Omega_{10})}\leq& C\big(\|H+M_h\|_{L^\iy(\Omega_{10})}+\|H+M_h\|_{L^2(\Omega_{10})}\big)\notag\\
\leq &C\big(\|H\|_{L^\iy(\Omega_{10})}+\sum^2_{j=1}\|M^{jj}_h\|_{L^\iy(\Omega_{10})}\big)\notag\\
\leq &C\big(\|H\|_{L^\iy(\Omega_{10})}+\int^{10}_{-10}\|h\|_{L^\iy(\Omega_{10})}ds\big)\notag\\
\leq &C\big(\|f\|_{L^\iy}+\|u\|_{L^\iy}\big).
\end{align}

Next, we define $M_g=(M^{ij}_g)_{3\times3}$ in the same way as in \eqref{sameway}. Similarly, it follows from the Fefferman-Stein's classical result on $BMO$-boundedness of singular integral operators \cite{Stein1993} that
\begin{align}
\|\nabla w\|_{BMO(\R^3)}\leq& C\big(\|G+M_g\|_{L^\iy(\R^3)}+\|G+M_g\|_{L^2(\R^3)}\big)\notag\\
\leq& C\big(\|f\|_{L^\iy}+\|f\|_{L^2}+\|u\|_{L^\iy}\big).
\end{align}

The proof is completed.
\end{proof}

The next lemma is a variant of the Brezis-Waigner inequality \cite{1980}, which will be used to derive the gradient estimate for density.
\begin{lemma}\label{BW}
Let $f\in W^{1,q}(\Omega)$ with $q\in(2,\iy)$. Then, there exists a positive constant $C$ depending on $q$ and $\Omega$ such that
\begin{align}
\|f\|_{L^\iy}\leq C\big(1+\|f\|_{BMO}\ln(e+\|\nabla f\|_{L^q})\big).
\end{align}
\end{lemma}
\begin{proof}
See \cite{SWZ2011}.
\end{proof}
The following lemma is due to Zlotnik \cite{Z2000}, which will be used to obtain the uniform (in time) upper bounds for the density.
\begin{lemma}\label{Z2000}
Let the function $y$ satisfy
\begin{align}
y^\prime(t)=g(y)+\bar{b}^\prime(t)\,\,\,\,\,\mathrm{on}\,\,\,[0,T],\,\,\,\,\,y(0)=y^0,
\end{align}
with $g\in C(\R)$ and $y, \bar{b}\in W^{1,1}(0,T)$. If $g(\iy)=-\iy$ and
\begin{align}
\bar{b}(t_2)-\bar{b}(t_1)\leq N_0+N_1(t_2-t_1)\label{N}
\end{align}
for all $0\leq t_1< t_2\leq T$ with some $N_0\geq 0$ and $N_1\geq 0$, then
\begin{align}
y(t)\leq max\{y^0,\bar{\xi}\}+N_0<\iy \,\,\,\,\,\,\,\mathrm{on}\,\,[0,T],\label{barxi}
\end{align}
where $\bar{\xi}$ is a constant such that
\begin{align}
g(\xi)\leq -N_1\,\,\,\,\mathrm{for}\,\,\,\xi\geq\bar{\xi}.
\end{align}
\end{lemma}
\begin{proof}
See \cite{Z2000}.
\end{proof}

\section{The case with large angular velocity}
In this section, we let $\Omega$ and $\a$ satisfy either of the two cases stated in the Theorem \ref{main}. We will first reduce the problem \eqref{eq}--\eqref{nab1} to \eqref{eqb} below, and then establish some necessary \textit{a priori} bounds for the smooth solutions to \eqref{eqb}. Finally, we will extend the local classical solution guaranteed by Proposition \ref{Local} to a global one.

\subsection{Reduction of the problem}
Inspired by \cite{H1999} and \cite{MN1992}, we first fix a scalar function $\varphi \in H^5(\R^3)$ as follows:
\begin{equation}
\varphi(x)=\varphi(|x|)=\left.
\begin{cases}
\exp^{-\ep^5(|x|-2)^5},\,\,\,|x|\geq 2, \\
1\,\,\,\,\,\,\,\,\,|x|< 2,
\end{cases}
\right.
\end{equation}
where $\ep$ is a positive constant.

Put
\begin{align}
b(t)=b(x,t)=-\frac{1}{2}\nabla\times\big(\varphi(x)\mid x\mid^2\vo(t)\big).\label{b}
\end{align}
It is easy to check that $b\in C\big([0,\iy);H^5(\R^3)\big)$, $\mathrm{div}b(t)=0$ in $\R^3$, and $b(t)=\vo(t)\times x$ near $\pt\Omega$ for all $t\geq0$.

Let
\begin{equation*}
u(x,t)=U(x,t)-b(x,t).
\end{equation*}
Then, the problem \eqref{eq}-\eqref{nab1} is reduced to the following:
\begin{equation}
\left.
\begin{cases}
\rho_{t}+(u+b-\vo\times x)\cdot\nabla \rho +\rho \div u=0, \\
\rho u_{t}+\rho(u+b-\vo\times x)\cdot\nabla u+\rho u\cdot \nabla b+\rho\vo\times u+\nabla p(\r)\\
-2\mu\div\D(u)-\lambda\nabla\textmd{div}u=\rho F(b)+\mu\triangle b,\,\,\,\,\,\,\,\,\,x\in \Omega,\,t\in (0,\iy),\\
u\,\,\,\,\, \mathrm{satisfies}\,\,\eqref{nab2}\,\,\,\,\mathrm{on}\,\,\partial\Omega,\\
\rho(0)=\rho_{0},\, u(0)=u_0\triangleq U_{0}-b(0), \,\,\,\,\,x\in\Omega,
\end{cases}\label{eqb}
\right.
\end{equation}
where $u_0(x)\triangleq U_0(x)-b(x,0)$ and $F(b)\triangleq(\vo\times x)\cdot\nabla b-b\cdot\nabla b-\vo\times b-b_t$.

It suffices to estimate the solution of \eqref{eqb}. Thus, in this section, let $T>0$ be a fixed time and let $(\r,u)$ be a smooth solution to \eqref{eqb} on $\Omega\times(0,T]$ in the class \eqref{localclass} with smooth initial data $(\r_0,u_0)$ satisfying
\begin{align}
&u_0\in D^1\cap D^2,\,\,\,\r_0\in L^1\cap H^2\cap W^{2,6},\,\,\,P(\r_0)\in H^2\cap W^{2,6},\,\,\,|x|\r_0^{\frac{1}{2}}\in L^\iy, \label{datab1}\\
&\|u_0\|_{D^1}+\|u_0\|_{D^2}\leq M,\,\,\,\,\,0\leq\inf\r_0\leq \sup\r_0\leq \br,\label{datab2}
\end{align}
and
\begin{equation}
-\mu\triangle \big(u_0+b(0)\big)-(\mu+\lambda)\nabla\div u_0+\nabla P(\r_0)=\r_0^{\frac{1}{2}}g\label{datab3}
\end{equation}
for some given constants $M>0$ and $\br\geq1$, and for some $g\in L^2$.

\subsection{\textit{A Priori} Estimates on the smooth solution}

To estimate $(\r,u)$, we set $\sg(t)=\min\{1,t\}$ and define
\begin{align}
A_1(T)\triangleq\sup_{t\in[0,T]}(\sg\|\nabla u\|^2_{L^2}),\,\,\,\,\,\,A_2(T)\triangleq\sup_{t\in[0,T]}\|\nabla u\|^2_{L^2}.
\end{align}
We denote by
\begin{align}
m_0\triangleq \int_{\Omega}\r_0dx
\end{align}
the initial mass.

In addition, we extend $\a$ to $\Omega$ such that $\a\in C^{\iy}_c(\overline{\Omega})$.

We have the following \textit{a priori} estimates on $(\r,u)$.
\begin{proposition}\label{priori3}
Suppose that $\o\in W^{4,1}(0,\iy)$, and suppose that \eqref{datab1}--\eqref{datab3} hold. Then, there exist positive constants  $K$ and $\e$ both depending only on $\mu$, $\l$, $a$, $\g$, $\w41$, $\br$, $\a$ and $M$ such that if $(\r,u)$ is a smooth solution to \eqref{eqb} on $\Omega\times(0,T]$ satisfying
\begin{align}
A_1(T)\leq 2\e^{\frac{1}{2}}, \,\,\,A_2(\sg(T))\leq 3K,\,\,\,\sup_{\Omega\times(0,T]}\r\leq 2\br, \label{condition}
\end{align}
then it holds that
\begin{align}
A_1(T)\leq \e^{\frac{1}{2}},\,\,\,A_2(\sg(T))\leq2K,\,\,\,\sup_{\Omega\times(0,T]}\r\leq \frac{7}{4}\br,\label{est0}
\end{align}
provided that $\ep\leq\e$ and $m_0\leq \ep^7$.
\end{proposition}
\begin{proof}
The Proposition \ref{priori3} is a easy consequence of Lemma \ref{EST3}, \ref{EST4} and \ref{density} below.
\end{proof}

Throughout this section, we use the convention that $C$ denotes a generic positive constant depending only on $\mu$, $\l$, $a$, $\g$, $\w41$, $\br$, $\a$ and $M$, but not on $g$ and $T$, and we write $C(\beta)$ to emphasize that $C$ depends on $\beta$. We also use the Einstein summation convention. Moreover, we assume that $m_0+\ep<1$ for simplicity.

We start with some observation on $b$. By a direct calculation, we have
\begin{align}
&\|b\|_{L^\iy}\leq C\ep^{-1}|\o(t)|,\,\,\,\,\,\,\|b_t\|_{L^\iy}\leq C\ep^{-1}|\frac{d\o}{dt}(t)|,\,\,\,\,\,\,
\|b_{tt}\|_{L^\iy}\leq C\ep^{-1}|\frac{d^2\o}{dt^2}(t)|,\label{be0}\\
&\|\nabla b\|_{L^\iy}\leq C|\o(t)|,\,\,\,\,\,\,\|\nabla b_t\|_{L^\iy}\leq C|\frac{d\o}{dt}(t)|,\,\,\,\,\,\,\|\nabla b_{tt}\|_{L^\iy}\leq C|\frac{d^2\o}{dt^2}(t)|,\label{be1}
\end{align}
and
\begin{align}
\|(|b|+|\cdot|)\nabla b\|_{L^\iy}\leq C\ep^{-1}|\o(t)|,\,\,\,\,\,\,\,\,\,\|(|b|+|\cdot|)\nabla b_t\|_{L^\iy}\leq C\ep^{-1}|\frac{d\o}{dt}(t)|.\label{bex}
\end{align}
Here, $\|(|b|+|\cdot|)\nabla b\|_{L^\iy}$ stands for the $L^\iy$-norm of $(|b|+|x|)\nabla b$.

Similarly, we have
\begin{align}
\|\nabla^\beta b\|_{L^p}\leq C_p\ep^{\b-1-\frac{1}{p}}|\o(t)|,\,\,\,\,\,\,\,\|\nabla^\beta b_t\|_{L^p}\leq C_p\ep^{\b-1-\frac{1}{p}}|\frac{d\o}{dt}(t)|\label{be2}
\end{align}
and
\begin{align}
\|(|b|+|\cdot|)\nabla^\beta b\|_{L^p}\leq C_p\ep^{\b-2-\frac{1}{p}}|\o(t)|,\,\,\,\,\,\,\,\|(|b|+|\cdot|)\nabla^\beta b_t\|_{L^p}\leq C_p\ep^{\b-2-\frac{1}{p}}|\frac{d\o}{dt}(t)|\label{be2x}
\end{align}
for $|\beta|=2,3,4,5$ and $1\leq p\leq \iy$, where $C_p$ is a positive constant depending only on $p$.

Next, we derive the conservation of mass and the energy estimate for $(\r,u)$.
\begin{lemma}\label{Energy}
Let $(\r,u)$ be a smooth solution to \eqref{eqb} on $\Omega\times(0,T]$ with $\sup_{\Omega\times(0,T]}\r\leq 2\br$. Suppose that \eqref{datab1}--\eqref{datab3} hold. Then, it holds that
\begin{align}
\|\r(t)\|_{L^1}=\|\r_0\|_{L^1}=m_0\label{mass}
\end{align}
for all $t\in[0,T]$.

Moreover, suppose that $\o\in W^{4,1}(0,\iy)$. Then, it holds that
\begin{align}
\sup_{0\leq t\leq T}(\|\sqrt{\r}u\|^2_{L^2}+\|P(\r)\|_{L^1})+\int_0^T\|\nabla u\|^2_{L^2}dt\leq C\ep^{\frac{1}{3}},\label{energy}
\end{align}
provided that $m_0\leq \ep^3$.
\end{lemma}
\begin{proof}
Integrating $\eqref{eqb}_1$ over $\Omega$, we obtain, after using the boundary condition $\eqref{eqb}_3$, that
\begin{align}
\frac{d}{dt}\int_{\Omega}\r dx=&\int_{\Omega}-\div\big(\r(u+b-\vo\times x)\big)\notag\\
=&-\lim_{R\rightarrow\iy}\int_{\pt B_R}\r(\vo\times x)\cdot \frac{x}{|x|}d\Gamma\notag\\
=&0,
\end{align}
which implies \eqref{mass}.

Sobolev's inequality yields that
\begin{align}
\|\o\|_{W^{k,\iy}(0,\iy)}\leq C\|\o\|_{W^{k+1,1}(0,\iy)}\leq C\label{wiy}
\end{align}
for $k=1,2,3$.

Add $\eqref{eqb}_1$ multiplied by $\frac{\g}{\g-1}\r^{\g-1}$ and $\eqref{eqb}_2$ multiplied by $u$. Integrating the resulting equality over $\Omega$, we obtain, after using the method of integration by parts, that
\begin{align}
&\frac{d}{dt}(\frac{1}{2}\|\sqrt{\r}u\|^2_{L^2}+\frac{1}{\g-1}\|P(\r)\|_{L^1})+2\mu\|\D(u)\|^2_{L^2}+\l\|\div u\|^2_{L^2}+2\mu|\sqrt{\a}u|^2_{L^2(\pt\Omega)}\notag\\
\leq& \|\r (u\cdot\nabla b)\cdot u\|_{L^1}+\|\r F(b)\cdot u\|_{L^1}+\mu\|\triangle b\cdot u\|_{L^1}\notag\\
\leq& \|\r^{\frac{1}{2}}\|_{L^\iy}\|\nabla b\|_{L^\iy}\|\sqrt{\r} u\|^2_{L^2}+\|\r^{\frac{1}{2}}\|_{L^2}\|F(b)\|_{L^\iy}\|\sqrt{\r}u\|_{L^2}
+C\|\triangle b\|_{L^\frac{6}{5}}\|u\|_{L^6}\notag\\
\leq& C(\br)|\o|\|\sqrt{\r}u\|^2_{L^2}+Cm^\frac{1}{2}_0\ep^{-1}(|\o|+|\frac{d\o}{dt}|)\|\sqrt{\r}u\|_{L^2}+C\ep^{\frac{1}{6}}|\o|\|\nabla u\|_{L^2},\notag\\
\leq& C(\br)|\o|\|\sqrt{\r}u\|^2_{L^2}+C(|\o|+|\frac{d\o}{dt}|)(\|\sqrt{\r}u\|^2_{L^2}+m_0\ep^{-2})+\delta\|\nabla u\|^2_{L^2}+C\frac{1}{\delta}\ep^{\frac{1}{3}}|\o|, \label{ene}
\end{align}
for any $\delta>0$, where we have used \eqref{be0}, \eqref{be1} and \eqref{wiy}, H\"{o}lder's inequality, Sobolev's inequality and Young's inequality, and the fact that
\begin{align}
\int_{\Omega}\r\big((b-\o\times x)\cdot\nabla u\big)\cdot udx=&\lim_{R\rightarrow\iy}\int_{\Omega_R}\frac{1}{2}\r(b-\vo\times x)\cdot\nabla(|u|^2)dx\notag\\
=&-\lim_{R\rightarrow\iy}\int_{\pt B_R}\frac{1}{2}\r|u|^2(\vo\times x)\cdot \frac{x}{|x|}d\Gamma\notag\\
=&0.
\end{align}
Note that
\begin{align}
\min\big\{2\mu,2\mu+3\l\big\}\|\D(u)\|_{L^2}\leq 2\mu\|\D(u)\|_{L^2}+\l\|\div u\|^2_{L^2}.\label{fact}
\end{align}
Integrating \eqref{ene} over $(0,t)$, we can obtain, after using the Lemma \ref{korn} and choosing $\delta$ suitably small, that
\begin{align}
&\frac{1}{2}\|\sqrt{\r}u\|^2_{L^2}+\frac{1}{\g-1}\|P(\r)\|_{L^1}+\frac{1}{C}\int^t_0\|\nabla u\|^2_{L^2}ds\notag\\
\leq&\int_{\Omega}\frac{1}{2}\r_0u^2_0+\frac{1}{\g-1}P(\r_0)dx+C\int^t_0(|\o|+|\frac{d\o}{dt}|)\|\sqrt{\r}u\|^2_{L^2}ds+C(m_0\ep^{-2}+\frac{1}{\delta}\ep^{\frac{1}{3}})\notag\\
\leq&C\|\r_0\|_{L^1}\|u_0\|^2_{L^\iy}+C(\br)\|\r_0\|_{L^1}+C\int^t_0(|\o|+|\frac{d\o}{dt}|)\|\sqrt{\r}u\|^2_{L^2}ds+C(m_0\ep^{-2}+\frac{1}{\delta}\ep^{\frac{1}{3}})\notag\\
\leq& C(M)(m_0+m_0\ep^{-2}+\frac{1}{\delta}\ep^{\frac{1}{3}})+C\int^t_0(|\o|+|\frac{d\o}{dt}|)\|\sqrt{\r}u\|^2_{L^2}ds,\label{gron}
\end{align}
where we have used \eqref{datab2}, \eqref{be0}, \eqref{wiy} and Sobolev's inequality.

Thus, \eqref{gron}, together with Gronwall's inequality, leads to \eqref{energy}, provided that $m_0\leq \ep^3$.

The proof is completed.
\end{proof}

In order to estimate $A_1(T)$ and $A_2(\sg(T))$, we first note that the normals on the boundary can be extended naturally to the interior as follows:
\begin{lemma}\label{expmap}
Let $l_0$ be the injectivity radius of the normal exponential map of $\pt\Omega$, i.e., the largest number such that the map
\begin{align}
\pt\Omega\times(-l_0,l_0)\,\,\rightarrow\,\,\Omega^{l_0}\triangleq\{x\in\R^3:\,\mathrm{dist}(x,\pt\Omega)<l_0\}:\,\,\,(\bar{x},l)\,\mapsto\,x=\bar{x}+l\n(\bar{x})\notag
\end{align}
is an injection.

Then, there exists a positive constant $C$ depending only on $\Omega$ such that $l_0\geq C>0$.
\end{lemma}
\begin{proof}
See \cite{CL2000}.
\end{proof}

Thus, there exists a vector field $\tn\in C^\iy_c(\overline{\Omega})$ such that $\tn|_{\pt\Omega}=\n$ and $\pt_{\tn}\tn=0$ on $\Omega^{l_0}$.

Next, we fix a sequence of local orthonormal tangent vector fields on $\pt\Omega$. There exists a covering of $\Omega$ under the form $\Omega\subset Q_0\cup^n_{i=1}Q_i$, where $\overline{Q_0}\subset\Omega$ and $Q_i\subset \Omega^{l_0}$ for each $i\geq 1$. Note that we can choose the size of each $Q_i$ $(1\leq i\leq n)$ as small as we want. Therefore, on each $\pt\Omega\cap Q_i$ $(1\leq i\leq n)$, there exist a local orthonormal tangent vector field $(\t_i,\io_i)$ and a local coordinate system $(y_1,y_2)$ (see \cite{Gray} for more details) such that there exists a smooth diffeomorphism
\begin{align}
\phi_i:\,\,D_i\subset \R^2\,\,\rightarrow\,\,\pt\Omega\cap Q_i:\,\,\,(y_1,y_2)\,\mapsto\,\phi_i(y_1,y_2),\notag
\end{align}
and it holds that
\begin{align}
&\t_i=(|\pt_{y_1}\phi_i|)^{-1}\pt_{y_1}\phi_i,\,\,\,\,\,\io_i=(|\pt_{y_2}\phi_i|)^{-1}\pt_{y_2}\phi_i,\,\,\,\,\,\t_i\cdot\io_i=0.
\end{align}

Similarly, by the Lemma \ref{expmap}, for $1\leq i\leq n$, we can extend $(\t_i,\io_i)$ to $\Omega\cap Q_i$ as follows: there exist $\tt_i,\tio_i\in C^\iy_c(\overline{\Omega})$ such that $\tt_i|_{\pt\Omega\cap Q_i}=\t_i$, $\pt_n\tt_i|_{\pt\Omega\cap Q_i}=0$, $\tio_i|_{\pt\Omega\cap Q_i}=\io_i$ and $\pt_n\tio_i|_{\pt\Omega\cap Q_i}=0$ (for more details see \eqref{coordinate}--\eqref{coordinate2} below).

In addition, there exists a partition of unit $\{\psi_i\}^n_{i=0}$ subordinate to the above covering of $\overline{\Omega}$ such that for $i=0,1,\cdots,n$, $\psi_i(x)\in C_c^\iy(\R^3)$, $\mathrm{supp}\psi_i\subset Q_i$, and $\Sigma^n_{i=0}\psi_i=1$ on $\overline{\Omega}$.

For simplicity, let
\begin{align}
f\triangleq u+b-\vo\times x.\label{fdef}
\end{align}
We now define the modified material derivative as
\begin{align}
\ud\triangleq u_t+f\cdot\nabla u+h,\label{uddef}
\end{align}
where
\begin{align}
h=h_1+h_2,\,\,\,\,\,
h_1\triangleq \big((u\cdot\nabla\tn)\cdot u\big)\tn,\,\,\,\,\,h_2\triangleq-(u\cdot\tn)\big[(u\cdot\nabla\tn)-\big((u\cdot\nabla\tn)\cdot\tn\big)\tn\big].\label{defh}
\end{align}

The following lemma is about some calculation which will be used later.
\begin{lemma}\label{calculation}
Let $(\r,u)$ be a smooth solution to \eqref{eqb} on $\Omega\times(0,T]$ with $\sup_{\Omega\times(0,T]}\r\leq 2\br$. Then, it holds that
\begin{align}
&(\pt_nu)_{\tau}=\theta(u)-2\a u_{\tau}\,\,\,\,\,\mathrm{on}\,\,\pt\Omega,\label{ntau}\\
&\ud\cdot \n=0\,\,\,\,\,\,\,\,\,\mathrm{on}\,\,\pt\Omega,\label{udot}\\
&\big(2\D(u)\n-(\nabla\times u)\times\n\big)_{\tau}=-2(\D(n)u)_{\tau}\,\,\,\,\,\mathrm{on}\,\,\pt\Omega,\label{rotcal}
\end{align}
and
\begin{align}
\div u&=\sum^n_{i=1}\psi_i\big(\pt_{\tau_i}(u\cdot\t_i)+\pt_{\iota_i}(u\cdot\io_i)\big)+\pt_n(u\cdot\tn)\,\,\,\,\,\mathrm{on}\,\,\pt\Omega.\label{divu}
\end{align}
Here, $\theta$ is the shape operator of the boundary that is given by
\begin{align}
\theta(u)=\nabla \tn u.
\end{align}

Moreover, $h$ satisfies the following equation:
\begin{align}
&\r h_t+\r f\cdot\nabla h-(2\mu+\l)\lap h=\sum^3_{i=1}I_i,\label{eqh}
\end{align}
where
\begin{align}
I_1=&-\big((H\cdot\nabla\tn)\cdot u\big)\tn-\big((u\cdot\nabla\tn)\cdot H\big)\tn,\label{I1}\\
I_2=&(H\cdot\tn)\big[(u\cdot\nabla\tn)-\big((u\cdot\nabla\tn)\cdot\tn\big)\tn\big],\label{I2}\\
I_3=&(u\cdot\tn)\big[(H\cdot\nabla\tn)-\big((H\cdot\nabla\tn)\cdot\tn\big)\tn\big],\label{I3}\\
I_4=&O(1)(|\nabla u|^2+|u|^3+|u|^2+|u|)1_{\{x\in\mathrm{supp}\tn\}},\label{I4}\\
H\triangleq&(\mu+\l)\nabla\times(\nabla\times u)-\nabla P(\r).
\end{align}
Here, and in the sequel, the Landau notation $O(1)$ is used to indicate a function whose absolute value remains uniformly bounded. For example,
\begin{align}
\r\big[\big((\vo\times u)\cdot\nabla\tn\big)\cdot u\big]\tn=O(1)|u|^2
\end{align}
means that, for some positive constant $C$, the following holds:
\begin{align}
\big|\r\big[\big((\vo\times u)\cdot\nabla\tn\big)\cdot u\big]\tn\big|\leq C|u|^2.
\end{align}
And $1_{\{x\in\mathrm{supp}\tn\}}$ is the characteristic function such that
\begin{equation}
1_{\{x\in\mathrm{supp}\tn\}}=\left.\begin{cases}
1,\,\,\,\,x\in \mathrm{supp}\tn,\\
0,\,\,\,\,\mathrm{otherwise}.
\end{cases}
\right.
\end{equation}
\end{lemma}

\begin{proof}
\eqref{ntau} follows from the boundary condition $\eqref{eqb}_3$ and the fact that
\begin{align}
(\nabla u\n)_{\tau}
=&\big(\nabla (u\cdot\tn)\big)_{\tau}-\nabla \tn u+\big(\n\cdot(\nabla\tn u)\big)\n\notag\\
=&-\nabla \tn u\,\,\,\,\,\,\,\,\mathrm{on}\,\,\,\pt\Omega.
\end{align}

Similarly, it holds that
\begin{align}
(u\cdot\nabla u)\cdot\n=u\cdot(\nabla u\n)=u\cdot\nabla(u\cdot\tn)-u\cdot(\nabla\tn u)=-h_1\cdot\tn\,\,\,\,\,\,\,\,\,\,\,\mathrm{on}\,\,\pt\Omega,
\end{align}
which implies $\eqref{udot}$.

\eqref{rotcal} follows from the Lemma 3.10 of \cite{XX2013}.

To get \eqref{divu}, it suffices to show
\begin{align}
\psi_i\div u=\psi_i\big(\pt_{\tau_i}(u\cdot\t_i)+\pt_{\iota_i}(u\cdot\io_i)+\pt_n(u\cdot\tn)\big)\,\,\,\,\,\,\,\,\,\,\mathrm{on}\,\,\pt\Omega\cap Q_i.\label{psidiv}
\end{align}
for $1\leq i\leq n$.

Let $(e_1,e_2,e_3)$ be the orthonormal canonical basis of $\R^3$. Noting that for $1\leq i\leq n$, $j=1,2,3$,
\begin{align}
\pt_j=&(e_j\cdot\t_i)\pt_{\tau_i}+(e_j\cdot\io_i)\pt_{\iota_i}+(e_j\cdot\n)\pt_n\notag\\
=&\t^j_i\pt_{\tau_i}+\io^j_i\pt_{\iota_i}+\n^j\pt_n\,\,\,\,\,\,\,\,\,\,\mathrm{on}\,\,\pt\Omega\cap Q_i,
\end{align}
and
\begin{align}
u^j=u\cdot e_j=(u\cdot\tt_i)\tt_i^j+(u\cdot\tio_i)\tio_i^j+(u\cdot\tn)\tn^j\,\,\,\,\,\,\,\,\,\,\mathrm{on}\,\,\pt\Omega\cap Q_i,
\end{align}
we can obtain \eqref{psidiv} after a direct calculation.

Using the Hodge decomposition, we rewrite the equation $\eqref{eqb}_2$ as
\begin{align}
&\r u_t+\r u\cdot\nabla u+ \nabla p(\r)-(2\mu+\l)\lap u-(\mu+\l)\nabla\times(\nabla\times u)+\rho u\cdot \nabla b+\rho\vo\times u \notag\\
=&\rho F(b)+\mu\triangle b.\label{eqhodge}
\end{align}
Using the equation \eqref{eqhodge} and noting the fact that $\tn$ has a compact support, we can obtain \eqref{eqh} after a direct calculation.

The proof is completed.
\end{proof}

Let
\begin{align}
G\triangleq\r\ud+\rho u\cdot \nabla b+\rho\vo\times u-\r F(b)-\mu\lap b-\r h.\label{G}
\end{align}

We now show the $L^p$-estimates for the vorticity $W\triangleq\nabla\times u$ and the effective viscous flux $F$ defined as
\begin{align}
F\triangleq(2\mu+\l)\div u-P(\r).\label{Fw}
\end{align}
\begin{lemma}\label{FW}
Let $(\r,u)$ be a smooth solution to \eqref{eqb} on $\Omega\times(0,T]$. Then, for $0\leq t\leq T$, $p\in [2,6]$, it holds that
\begin{align}
\|F\|_{L^2}+\|W\|_{L^2}\leq& C(\|\nabla u\|_{L^2}+\|P(\r)\|_{L^2}),\label{FL2}\\
\|\nabla F\|_{L^p}+\|\nabla W\|_{L^p}\leq& C_p(\|G\|_{L^p}+\|\nabla u\|_{L^p}+\|\nabla u\|_{L^2}),\label{FW2}\\
\|F\|_{L^p}+\|W\|_{L^p}\leq& C_p(\|G\|_{L^2}+\|\nabla u\|_{L^2}+\|P(\r)\|_{L^2}),\label{FW3}
\end{align}
and
\begin{align}
\|\nabla u\|_{L^p}\leq& C_p(\|G\|_{L^2}+\|\nabla u\|_{L^2}+\|P(\r)\|_{L^2}+\|P(\r)\|_{L^p}),\label{FW4}
\end{align}
where $C_p$ is a positive constant depending only on $p$, $\mu$, $\l$ and $\Omega$.
\end{lemma}
\begin{proof}
\eqref{FL2} follows from \eqref{Fw} directly.

Then, noting that $u$ satisfies the elliptic boundary value problem,
\begin{equation}
\left.
\begin{cases}
(\mu+\l)\triangle u=\nabla F-(\mu+\l)\nabla\times W+\nabla P(\r)\,\,\,\,\,\,\,\mathrm{in}\,\,\Omega,\\
u\,\,\mathrm{satisfies}\,\,\,\eqref{nab2}\,\,\,\,\,\,\,\,\,\,\,\mathrm{on}\,\,\,\,\pt\Omega,
\end{cases}\label{decom}
\right.
\end{equation}
we obtain, after applying the Lemma \ref{sob} and \ref{exterior} on the equation \eqref{decom}, that
\begin{align}
\|\nabla u\|_{L^p}\leq &C_p(\|F\|_{L^p}+\|W\|_{L^p}+\|P(\r)\|_{L^p}+\|u\|_{L^6})\notag\\
\leq &C_p(\|F\|_{L^p}+\|W\|_{L^p}+\|P(\r)\|_{L^p}+\|\nabla u\|_{L^2}).\label{FW1}
\end{align}

In order to prove \eqref{FW2}, via \eqref{rotcal}, we rewrite the equation $\eqref{eqb}_2$ as
\begin{equation}
\left.
\begin{cases}
\mu\triangle W=\nabla\times G \,\,\,\,\,\,\,\mathrm{in}\,\,\Omega,\\
\big(2\D(u)\n-W\times\n\big)_{\tau}=-2(\D(n)u)_{\tau}\,\,\,\,\,\mathrm{on}\,\,\pt\Omega.
\end{cases}\label{eq3}
\right.
\end{equation}
If we multiply $\psi_m$ $(0\leq m\leq n)$ on both sides of \eqref{eq3}, then the equation \eqref{eq3} is reduced to
\begin{equation}
\left.
\begin{cases}
\mu\triangle (\psi_mW)=\nabla\times(\psi_mG)+O(1)1_{\mathrm{supp}\psi_m}(|\nabla u|+|G|)\,\,\,\,\,\,\,\mathrm{in}\,\,\Omega\cap Q_m,\\
\big(2\psi_m\D(u)\cdot\n-\psi_mW\times\n\big)_{\tau}=-2(\psi_m\D(n)\cdot u)_{\tau}\,\,\,\,\,\mathrm{on}\,\,\pt\Omega\cap Q_m.
\end{cases}\label{eq4}
\right.
\end{equation}
Thus, it suffices to estimate $\|\nabla (\psi_mW)\|_{L^p(\Omega\cap Q_m)}$.

Provided that the size of $Q_m$ $(1\leq m\leq n)$ was chosen suitably small, it is convenient to use the local coordinates $x=x(y_1,y_2,z)=\Phi_m(y_1,y_2,z)$ in $Q_m$, where $\Phi_m$ is a smooth diffeomorphism from some bounded subset $\widetilde{D_m}\subset D_m\times(-l_0,l_0)\subset \R^3$ to $Q_m$ such that
\begin{align}
\Phi_m:\,(y_1,y_2,z)\mapsto \phi_m(y_1,y_2)+z\n(\phi_m(y_1,y_2)). \label{coordinate}
\end{align}
A local basis in $\Omega\cap Q_m$ is indeed given by $(\tt_m,\tio_m,\tn)$ satisfying
\begin{align}
&(\tt_m,\tio_m,\tn)=\big(|(\pt_{y_1}\Phi_m|)^{-1}\pt_{y_1}\Phi_m,(|\pt_{y_2}\Phi_m|)^{-1}\pt_{y_1}\Phi_m,\pt_z\Phi_m\big).\label{coordinate2}
\end{align}

In the following we omit the subscript $m$ for notational convenience. We use the notations $y=(y_1,y_2)$, $\tnabla_y=(\pt_{y_1},\pt_{y_2})$, $\tnabla=(\pt_{y_1},\pt_{y_2},\pt_{z})$ and $(\pt_{\tt},\pt_{\tio},\pt_{\tn})=(|\pt_{y_1}\Phi|^{-1}\pt_{y_1},|\pt_{y_2}\Phi|^{-1}\pt_{y_1},\pt_z)$, and let
\begin{align}
\tu(y,z)=\big(\tu^{\tt},\tu^{\tio},\tu^{\tn}\big)=\big(u^{\tt}(\Phi(y,z),u^{\tio}(\Phi(y,z),u^{\tn}(\Phi(y,z)\big),\notag\\
\tg(y,z)=\big(\tg^{\tt},\tg^{\tio},\tg^{\tn}\big)=\big(G^{\tt}(\Phi(y,z),G^{\tio}(\Phi(y,z),G^{\tn}(\Phi(y,z)\big),\notag\\
\tw(y,z)=(\tw^{\tio,\tn},\tw^{\tn,\tt},\tw^{\tt,\tio})=(\pt_{\tio}\tu^{\tn}-\pt_{\tn}\tu^{\tio},\pt_{\tn}\tu^{\tt}-\pt_{\tt}\tu^{\tn},\pt_{\tt}\tu^{\tio}-\pt_{\tio}\tu^{\tt}),
\end{align}
where $u^{\tt}=u\cdot\tt$, $u^{\tio}=u\cdot\tio$, $u^{\tn}=u\cdot\tn$.

Then, by using the local (in $Q$) decompositions such as
\begin{align}
\pt_1=&e_1\cdot\nabla=(e_1\cdot\tt)\tt\cdot\nabla+(e_2\cdot\tio)\tio\cdot\nabla+(e_1\cdot\tn)\tn\cdot\nabla\notag\\
=&\tt^1\pt_{\tt}+\tio^1\pt_{\tio}+\tn^1\pt_{\tn},\label{pt1}
\end{align}
and
\begin{align}
u^{1}=u\cdot e_1=(u^{\tt}\tt+u^{\tio}\tio+u^{\tn}\tn)\cdot e_1=\tu^{\tt}\tt^1+\tu^{\tio}\tio^1+\tu^{\tn}\tn^1,
\end{align}
we can obtain, after a direct calculation, that
\begin{align}
W(\Phi(y,z))=J(y,z)\tw(y,z)+\tu^{\tn}j_{\tn}+\tu^{\tt}j_{\tt}+\tu^{\tio}j_{\tio},
\end{align}
where
\begin{align}
&j_{\tn}=\tio\times\pt_{\tio}\tn+\tt\times\pt_{\tt}\tn,\,\,\,\,j_{\tt}=\tio\times\pt_{\tio}\tt+\tt\times\pt_{\tt}\tt+\tn\times\pt_{\tn}\tt,\\
&j_{\tio}=\tio\times\pt_{\tio}\tio+\tt\times\pt_{\tt}\tio+\tn\times\pt_{\tn}\tio,
\end{align}
and
\begin{align}
J=(\tt,\tio,\tn)=
\left(
  \begin{array}{ccc}
    \tt^1 & \tio^1 & \tn^1 \\
    \tt^2 & \tio^2 & \tn^2 \\
    \tt^3 & \tio^3 & \tn^3 \\
  \end{array}
\right).
\end{align}

We claim that $j_{\tn}=0$. Indeed,
\begin{align}
\pt_{\tio}\tn=&|\pt_{y_2}\Phi|^{-1}\pt_{y_2}\tn=|\pt_{y_2}\Phi|^{-1}\big((\pt_{y_2}\tn\cdot\tt)\tt+(\pt_{y_2}\tn\cdot\tio)\tio\big)\notag\\
=&|\pt_{y_2}\Phi|^{-1}\big(|\pt_{y_1}\Phi|^{-1}(\pt_{y_2}\tn\cdot\pt_{y_1}\Phi)\tt+|\pt_{y_2}\Phi|^{-1}(\pt_{y_2}\tn\cdot\pt_{y_2}\Phi)\tio\big),
\end{align}
which leads to
\begin{align}
\tio\times\pt_{\tio}\tn=|\pt_{y_2}\Phi|^{-1}|\pt_{y_1}\Phi|^{-1}(\pt_{y_2}\tn\cdot\pt_{y_1}\Phi)\tio\times\tt.
\end{align}
Similarly, we have
\begin{align}
\tt\times\pt_{\tt}\tn=|\pt_{y_2}\Phi|^{-1}|\pt_{y_1}\Phi|^{-1}(\pt_{y_1}\tn\cdot\pt_{y_2}\Phi)\tt\times\tio.
\end{align}
Noting that $\pt_{y_2}\tn\cdot\pt_{y_1}\Phi$ and $\pt_{y_1}\tn\cdot\pt_{y_2}\Phi$ are the same coefficient of the second fundamental form of $\Phi$, we get
\begin{align}
j_{\tn}=|\pt_{y_2}\Phi|^{-1}|\pt_{y_1}\Phi|^{-1}(\pt_{y_2}\tn\cdot\pt_{y_1}\Phi)(\tio\times\tt+\tt\times\tio)=0,
\end{align}
which also implies
\begin{align}
W(\Phi(y,z))=J(y,z)\tw(y,z)+\tu^{\tt}j_{\tt}+\tu^{\tio}j_{\tio}.\label{change}
\end{align}
Thus, $\tw(y,z)$ can be expressed as follows: for example,
\begin{align}
\tw^{\tt,\tio}(y,z)=\tn\cdot W(\Phi(y,z))-\big((\tt\times\tio)\cdot\tn\big)\big[(\tio\cdot\pt_{\tt}\tt)\tu^{\tt}-(\tt\cdot\pt_{\tio}\tio)\tu^{\tio}\big].\label{wtl}
\end{align}

Then, since
\begin{align}
\lap \varphi=(\pt^2_{\tt}+\pt^2_{\tio}+\pt^2_{\tn})\varphi+(\tt\cdot\pt_{\tt}\tn+\tio\cdot\pt_{\tio}\tn)\pt_{\tn}\varphi+O(1)|\tnabla_y \varphi|\,\,\,\,\,\,\mathrm{in}\,\,\,\Omega\cap Q_m   \label{lap}
\end{align}
for any smooth function $\varphi$, we can obtain from \eqref{change} and $\eqref{eq4}$ dot-multiplied by $\tt$ that
\begin{equation}
\left.
\begin{cases}
\big((\pt_{y_1}\Phi)^{-2}\pt^2_{y_1}+(\pt_{y_2}\Phi)^{-2}\pt^2_{y_2}+\pt^2_z\big)(\tpsi\tw^{\tio,\tn})\\
=\nabla_y\cdot\big(O(1)(|G|+|\tnabla \tu|)\big)+O(1)(|\tnabla \tu|+|\tu|+|G|)\,\,\,\,\,\,\,\mathrm{in}\,\,\Omega\cap Q_m,\\
\tpsi\tw^{\tio,\tn}=O(1)\tu \,\,\,\,\,\mathrm{on}\,\,\pt\Omega\cap Q_m,
\end{cases}\label{eq5}
\right.
\end{equation}
where $\tpsi(y,z)=\psi(\Phi(y,z))$.

Therefore, by an analogous argument as in the proof of the Lemma \ref{exterior}, we can obtain that
\begin{align}
&\|\tnabla(\tpsi\tw^{\tio,\tn})\|_{L^p(\Phi^{-1}(\Omega\cap Q))}\notag\\
\leq& C_p(\|\tg\|_{L^p\big(\Phi^{-1}(\Omega\cap Q)\big)}+\|\tnabla \tu\|_{L^p(\Phi^{-1}\big(\Omega\cap Q)\big)}+\|\tu\|_{L^6(\Phi^{-1}\big(\Omega\cap Q)\big)}). \label{together}
\end{align}

A similar argument applies to $\tpsi\tw^{\tn,\tt}$.

To obtain the estimate of $\tpsi\tw^{\tt,\tio}$, we get from $\eqref{wtl}$, $\eqref{lap}$ and $\eqref{eq4}$ dot-multiplied by $\tn$ that
\begin{align}
&\big((\pt_{y_1}\Phi)^{-2}\pt^2_{y_1}+(\pt_{y_2}\Phi)^{-2}\pt^2_{y_2}\big)(\tpsi\tw^{\tt,\tio})\notag\\
=&\pt_{\tt}\tg^{\tio}-\pt_{\tio}\tg^{\tt}-\pt^2_{\tn}(\tpsi\tw^{\tt,\tio})
-\pt^2_{\tn}\big\{\tpsi\big((\tt\times\tio)\cdot\tn\big)\big[(\tio\cdot\pt_{\tt}\tt)\tu^{\tt}-(\tt\cdot\pt_{\tio}\tio)\tu^{\tio}\big]\big\}    \notag\\
&+\tpsi(\tt\cdot\pt_{\tt}\tn+\tio\cdot\pt_{\tio}\tn)\pt_{\tn}\tw^{\tt,\tio}+\tnabla_y\cdot(O(1)|\tnabla u|)+O(1)(|G|+|\tnabla u|+|\tu|), \label{since0}
\end{align}
where we have used the fact that
\begin{align}
\tn\cdot\pt^2_{\tn}W=\pt^2_{\tn}(\tn\cdot W).
\end{align}
In order to estimate the terms on the right hand side of \eqref{since0} which may contain derivatives in the direction $z$, we rewrite $\pt_{\tn}\tw^{\tt,\tio}$, $\pt_{\tn}\tu^{\tt}$ and $\pt_{\tn}\tu^{\tio}$ as follows:
\begin{align}
\pt_{\tn}\tu^{\tt}=&\tw^{\tn,\tt}+\pt_{\tt}\tu^{\tn},\,\,\,\pt_{\tn}\tu^{\tio}=-\tw^{\tio,\tn}+\pt_{\tio}\tu^{\tn},\label{since1}\\
\pt_{\tn}\tw^{\tt,\tio}=&-\pt_{\tio}\tw^{\tn,\tt}-\pt_{\tt}\tw^{\tio,\tn}+(\pt_{\tn}\pt_{\tt}-\pt_{\tt}\pt_{\tn})\tu^{\tio}\notag\\
&-(\pt_{\tn}\pt_{\tio}-\pt_{\tio}\pt_{\tn})\tu^{\tt}+(\pt_{\tt}\pt_{\tio}-\pt_{\tio}\pt_{\tt})\tu^{\tn}. \label{since2}
\end{align}
Hence, we can obtain from \eqref{since0}, \eqref{since1} and \eqref{since2} that
\begin{align}
&\big((\pt_{y_1}\Phi)^{-2}\pt^2_{y_1}+(\pt_{y_2}\Phi)^{-2}\pt^2_{y_2}\big)(\tpsi\tw^{\tt,\tio})\notag\\
=&\tnabla_y\cdot\big(O(1)(|G|+|\tnabla u|)\big)+O(1)(|\tnabla\tw^{\tio,\tn}|+|\tnabla\tw^{\tn,\tt}|+|G|+|\tnabla u|+|u|),
\end{align}
where we have used the fact that $\pt_{\tn}\pt_{\tt}=\pt_{\tt}\pt_{\tn}+(\pt_{\tn}\pt_{\tt}-\pt_{\tt}\pt_{\tn})$ to deal with the terms which contain second order derivatives in the direction $z$.

Thus, for any fixed $z^{\prime}$, we can estimate $\|\tnabla_y(\tpsi\tw^{\tt,\tio})\|_{L^p\big(\Phi^{-1}(\Omega\cap Q)\cap \{z=z^{\prime}\}\big)}$. Integrating this bound with respect to $z$, and applying \eqref{together}, we obtain that $\|\tnabla_y(\tpsi\tw^{\tt,\tio})\|_{L^p\big(\Phi^{-1}(\Omega\cap Q)\big)}$ is bounded by the right hand side of \eqref{together}. The bound for $\|\pt_z(\tpsi\tw^{\tt,\tio})\|_{L^p\big(\Phi^{-1}(\Omega\cap Q)\big)}$ follows from \eqref{since2}.

Therefore, we have
\begin{align}
&\|\tnabla(\tpsi\tw)\|_{L^p\big(\Phi^{-1}(\Omega\cap Q)\big)}\notag\\
\leq& C_p(\|\tg\|_{L^p\big(\Phi^{-1}(\Omega\cap Q)\big)}+\|\tnabla \tu\|_{L^p(\Phi^{-1}\big(\Omega\cap Q)\big)}+\|\tu\|_{L^6(\Phi^{-1}\big(\Omega\cap Q)\big)}).\label{norm}
\end{align}
Note that the $L^p$-norm is an equivalent norm under the coordinate transformation \eqref{coordinate}. Thus, \eqref{norm}, together with \eqref{change} and Sobolev's inequality, implies
\begin{align}
\|\nabla(\psi_m W)\|_{L^p}\leq& C_p(\|G\|_{L^p}+\|\nabla u\|_{L^p}+\|\nabla u\|_{L^2})\label{mW}
\end{align}
for $1\leq m\leq n$.

Similarly, the bound for $\|\nabla(\psi_0 W)\|_{L^p}$ can be obtained by an analogous argument as in the proof of the Lemma \ref{exterior} such that \eqref{mW} holds also for $m=0$.

Moreover, $F$ satisfies
\begin{align}
\nabla F=G+\mu\nabla\times W,
\end{align}
which, together with \eqref{mW}, gives \eqref{FW2}.

In addition, it follows from \eqref{FL2}--\eqref{FW2} and Sobolev's inequality that
\begin{align}
&\|F\|_{L^p}\leq C_p(\|F\|_{L^2}+\|\nabla F\|_{L^2})\leq C(\|G\|_{L^2}+\|\nabla u\|_{L^2}+\|P(\r)\|_{L^2})
\end{align}
and
\begin{align}
&\|W\|_{L^p}\leq C_p(\|\nabla W\|_{L^2}+\|\nabla u\|_{L^2})\leq C_p(\|G\|_{L^2}+\|\nabla u\|_{L^2}),
\end{align}
the combination of which leads to \eqref{FW3}.

\eqref{FW4} is a direct result from \eqref{FW3} and \eqref{FW1}.

The proof is completed.
\end{proof}

Inspired by \cite{Hoff1995}, we now derive the preliminary $L^2$-bounds for $\nabla u$ and $\r\ud$ as follows:
\begin{lemma}\label{Lemma1}
Let $(\r,u)$ be a smooth solution to \eqref{eqb} on $\Omega\times(0,T]$ with $\sup_{\Omega\times(0,T]}\r\leq 2\br$. Suppose that $\o\in W^{4,1}(0,\iy)$ and $m_0\leq\ep^3$. Then, there exists a positive constant $C$ such that
\begin{align}
&\big(\mu\eta\|\D(u)\|^2_{L^2}+\frac{\l}{2}\eta\|\div u\|^2_{L^2}\big)_t+\mu\big(\eta|\sqrt{\a}\ud|^2_{L^2(\pt\Omega)}\big)_t+\frac{\eta}{2}\|\sqrt{\r}u\|^2_{L^2}\notag\\
\leq &\big(\eta\int_\Omega P(\r)\div udx\big)_t+\mu\big(\eta\int_{\Omega}\lap b\cdot udx\big)_t+C|\eta^\prime|(\ep^{\frac{1}{3}}+\|\nabla u\|^2_{L^2})\notag\\
&+C\eta\big(\|\nabla u\|^2_{L^2}+\|\nabla u\|^3_{L^2}+\|\nabla u\|^6_{L^2}+\ep^{\frac{1}{3}}\big),\label{est1}
\end{align}
and
\begin{align}
&\big(\eta\|\sqrt{\r}\ud\|^2_{L^2}\big)_t+\eta\|\nabla\ud\|^2_{L^2}+\eta|\sqrt{\a}\ud|^2_{L^2(\pt\Omega)} \notag\\
\leq& C\eta(|\o|+|\frac{d\o}{dt}|+|\frac{d^2\o}{dt^2}|)\|\sqrt{\r}\ud\|^2_{L^2}+C|\eta^\prime|\|\sqrt{\r}\ud\|^2_{L^2}\notag\\
&+C\eta\|\nabla u\|_{L^2}\|\r\ud\|^3_{L^2}
+C\eta\big(\|\nabla u\|^2_{L^2}+\|\nabla u\|^4_{L^2}+\ep^{\frac{1}{3}}\big),\label{est2}
\end{align}
where $\eta=\eta(t)\geq 0$ is a piecewise smooth function.
\end{lemma}
\begin{proof}
Multiplying $\eqref{eqb}_2$ by $\eta\ud$, we obtain, after integrating the resulting equality over $\Omega$, that
\begin{align}
\int_{\Omega}\eta\r|\ud|^2dx=&\int_{\Omega}\big(-\eta\nabla P(\r)\cdot \ud+2\mu\eta\div\D(u)\cdot\ud+\l\eta\nabla\div u\cdot\ud\notag\\
&-\r (u\cdot\nabla b)\cdot\ud-\r\vo\times u\cdot\ud+\r F(b)\cdot\ud+\mu\lap b\cdot\ud-\r h\cdot\ud\big)dx\notag\\
\triangleq&\sum^8_{i=1}J_i.\label{J0}
\end{align}

We now use the method of integration by parts to estimate each term on the right hand side of \eqref{J0} as follows:
\begin{align}
J_1=&-\int_{\Omega}\eta\nabla P(\r)\cdot\ud dx\notag\\
=&\big(\int_{\Omega}\eta P(\r)\div udx\big)_t-\eta^\prime\int_{\Omega}P(\r)\div udx-\eta\int_{\Omega}P^\prime(\r)\r_t\div udx\notag\\
&-\eta\int_{\Omega}(f\cdot\nabla u)\cdot\nabla P(\r) dx+\int_{\Omega}\eta P(\r)\div hdx\notag\\
=&\big(\int_{\Omega}\eta P(\r)\div udx\big)_t-\eta^\prime\int_{\Omega}P(\r)\div udx+\eta\int_{\Omega}P^\prime(\r)\r(\div u)^2dx\notag\\
&+\eta\int_{\Omega}P(\r)\big[\div(f\cdot\nabla u)-\div(f\div u)\big]dx+\eta\int_{\Omega} P(\r)\div hdx\notag\\
=&\big(\int_{\Omega}\eta P(\r)\div udx\big)_t-\eta^\prime\int_{\Omega}P(\r)\div udx+\eta\int_{\Omega}P^\prime(\r)\r(\div u)^2+P(\r)(\partial_if^j\partial_ju^i+\div h)dx\notag\\
\leq &(\int_{\Omega}\eta P(\r)\div udx)_t+C(\br)(\eta+|\eta^\prime|)\|\nabla u\|^2_{L^2}+C(\br)|\eta^\prime| m^\frac{1}{2}_0,\label{J1}
\end{align}
where we have used the equation $\eqref{eqb}_1$, \eqref{be1}, \eqref{mass}, \eqref{wiy}, \eqref{udot} and the fact that
\begin{align}
\|u\|_{H^1(\mathrm{supp}h)}\leq C\|\nabla u\|_{L^2(\mathrm{supp}\tn)}\leq C\|\nabla u\|_{L^2}.
\end{align}

Integration by parts leads to
\begin{align}
J_2=&2\mu\int_{\Omega}\eta\div\D(u)\cdot\ud dx\notag\\
=&2\mu\int_{\pt\Omega}\eta\ud\cdot(\D(u)\n)d\Gamma-\mu\int_{\Omega}\eta(\pt_i u^j+\pt_j u^i)\big(\pt_i u^j_t+\pt_i(f^k\pt_k u^j)+\pt_ih^j\big)dx\notag\\
=&-\mu\big(\eta\int_{\pt\Omega}\a|u|^2d\Gamma\big)_t+\mu\eta^\prime\int_{\pt\Omega}\a|u|^2d\Gamma
+2\mu\eta\int_{\pt\Omega}\a (u\cdot\nabla u)\cdot ud\Gamma\notag\\
&-\mu(\eta\|\D(u)\|^2_{L^2})_t+\mu\eta^\prime\|\D(u)\|^2_{L^2}-\mu\eta\int_{\Omega}(\pt_i u^j+\pt_j u^i)\big(\pt_if^k\pt_k u^j+\pt_ih^j\big)dx\notag\\
&-\mu\eta\int_{\pt\Omega}f^k\n^k(\pt_i u^j+\pt_j u^i)\pt_iu^jd\Gamma+\frac{1}{2}\mu\eta\int_{\Omega}(\pt_i u^j+\pt_j u^i)\pt_kf^k\pt_iu^jdx\notag\\
\leq&-\mu\big(\eta\|\D(u)\|^2_{L^2}\big)_t-\mu\big(\eta|\sqrt{\a}u|^2_{L^2(\pt\Omega)}\big)_t+C(\eta+|\eta^\prime|)\|\nabla u\|^2_{L^2}+C\eta\|\nabla u\|^3_{L^3},\label{J2}
\end{align}
where we have used the inequality obtained from that the Fourier transform and the trace theorem (see \cite{Evans}), and Sobolev's inequality and Young's inequality to deduce that
\begin{align}
&\mu\eta^\prime\int_{\pt\Omega}\a|u|^2d\Gamma+2\mu\eta\int_{\pt\Omega}\a (u\cdot\nabla u)\cdot ud\Gamma\notag\\
=&\mu\eta^\prime\int_{\pt\Omega}\a|u|^2d\Gamma
+2\mu\eta\sum^n_{i=1}\int_{\pt\Omega\cap Q_i}\a \psi_i\big[(u\cdot\t_i)(\pt_{\tau_i}u\cdot u)+(u\cdot\io_i)(\pt_{\io_i}u\cdot u)\big]d\Gamma\notag\\
\leq& C|\eta^\prime|\|u\|^2_{H^{\frac{1}{2}}(\pt\Omega)}+C\eta\sum^n_{i=1}\|u\|_{H^{\frac{1}{2}}(\pt\Omega\cap Q_i)}\|u\cdot u\|_{H^{\frac{1}{2}}(\pt\Omega\cap Q_i)}\notag\\
\leq& C|\eta^\prime|\|u\|^2_{H^1(\Omega_5)}+C\eta\|u\|_{H^1(\Omega_5)}\|u\cdot u\|_{H^1(\Omega_5)}\notag\\
\leq& C|\eta^\prime|\|\nabla u\|^2_{L^2}+C\eta\|\nabla u\|_{L^2}(\|u\|^2_{L^4(\Omega_5)}+\|u\|_{L^6(\Omega_5)}\|\nabla u\|_{L^3(\Omega_5)})\notag\\
\leq& C|\eta^\prime|\|\nabla u\|^2_{L^2}+C\eta\|\nabla u\|^3_{L^2}+C\eta\|\nabla u\|^3_{L^3}.
\end{align}
Similarly, we have
\begin{align}
J_3=&\l\int_{\Omega}\eta\nabla\div u\cdot\ud dx\notag\\
=&-\frac{\l}{2}(\eta\|\div u\|^2_{L^2})_t+\frac{\l\eta^\prime}{2}\|\div u\|^2_{L^2}-\l\eta\int_{\Omega}\div u\big(\div(f\cdot\nabla u)+\div h\big)dx\notag\\
\leq&-\frac{\l}{2}(\eta\|\div u\|^2_{L^2})_t+C(\eta+|\eta^\prime|)\|\nabla u\|^2_{L^2}+C\eta\|\nabla u\|^3_{L^3}.\label{J3}
\end{align}
We apply Sobolev's inequality and Young's inequality and use \eqref{energy} to show that
\begin{align}
J_4+J_5+J_6+J_8
\leq& \eta\|\nabla b\|_{L^\iy}\|\sqrt{\r}u\|_{L^2}\|\sqrt{\r}\ud\|_{L^2}
+\eta|\o|\|\sqrt{\r}u\|_{L^2}\|\sqrt{\r}\ud\|_{L^2}\notag\\
&+\eta\|\r^{\frac{1}{2}}\|_{L^2}\|F(b)\|_{L^\iy}\|\sqrt{\r}\ud\|_{L^2}+C\eta\|\r^{\frac{1}{2}}\|_{L^\iy}\|u\|^2_{L^4(\mathrm{supp}\tn)}\|\sqrt{\r}\ud\|_{L^2}\notag\\
\leq& \frac{1}{4}\eta\|\sqrt{\r}\ud\|^2_{L^2}+C(\br)\eta\|\nabla u\|^4_{L^2}+C\eta\ep^{\frac{1}{3}}+C\eta m_0\ep^{-2}(|\o|+|\frac{d\o}{dt}|),
\end{align}
and
\begin{align}
J_7=&\mu\int_{\Omega}(\eta\lap b\cdot u)_t-\eta^\prime\lap b\cdot u-\eta\lap b_t\cdot u+\eta \big((u\cdot\nabla u)+h\big)\cdot\lap b
+\eta \big((b-\vo\times x)\cdot\nabla u\big)\cdot\lap b dx\notag\\
=&\mu\int_{\Omega}(\eta\lap b\cdot u)_t-\eta^\prime\lap b\cdot u-\eta\lap b_t\cdot u+\eta \big((u\cdot\nabla u)+h\big)\cdot \lap b
-\eta(b-\vo\times x)^i u^j\pt_i\lap b^j dx\notag\\
\leq&\mu\big(\eta\int_{\Omega}\lap b\cdot udx\big)_t+C|\eta^{\prime}|\|\lap b\|_{L^\frac{6}{5}}\|u\|_{L^6}+C\eta\|\lap b_t\|_{L^\frac{6}{5}}\|u\|_{L^6}\notag\\
&+C\eta\|\lap b\|_{L^3}\|u\|_{L^6}\|\nabla u\|_{L^2}+C\eta\|\lap b\|_{L^\iy}\|u\|^2_{L^2(\mathrm{supp}\tn)}
+C\eta\|(|b|+|\cdot|)\nabla^3 b\|_{L^{\frac{6}{5}}}\|u\|_{L^6}\notag\\
\leq&\mu\big(\eta\int_{\Omega}\lap b\cdot udx\big)_t+C|\eta^{\prime}|(\|\nabla u\|^2_{L^2}
+\ep^{\frac{1}{3}}|\o|)+C\eta\|\nabla u\|^2_{L^2}+C\eta\ep^{\frac{1}{3}}(|\o|+|\frac{d\o}{dt}|),\label{J7}
\end{align}
where we have used \eqref{be0}--\eqref{mass}, \eqref{wiy} and the method of integration by parts.

Moreover, it follows from \eqref{G}, \eqref{be0}, \eqref{be1}, \eqref{be2}, \eqref{mass} and Holder's inequality and Sobolev's inequality that
\begin{align}
\|G\|_{L^2}\leq&C\big(\|\r^{\frac{1}{2}}\|_{L^2}\|\sqrt{\r}\ud\|_{L^2}+\|\r^{\frac{1}{2}}\|_{L^\iy}\|\nabla b\|_{L^\iy}\|\sqrt{\r}u\|_{L^2}
+|\o|\|\sqrt{\r}u\|_{L^2}\notag\\
&+\|\r\|_{L^2}\|F(b)\|_{L^\iy}+\|\lap b\|_{L^2}+\|u\|^2_{L^6}\big)\notag\\
\leq&C\big(\|\sqrt{\r}\ud\|_{L^2}+|\o|\|\sqrt{\r}u\|_{L^2}+\ep^{\frac{1}{2}}(|\o|+|\frac{d\o}{dt}|)+\|\nabla u\|^2_{L^2}\big),\label{estG}
\end{align}
which, together with \eqref{FW4}, \eqref{energy} and \eqref{wiy}, and Sobolev's inequality, Holder's inequality and Young's inequality, gives that
\begin{align}
\|\nabla u\|^3_{L^3}\leq& \|\nabla u\|^\frac{3}{2}_{L^2}\|\nabla u\|^\frac{3}{2}_{L^6}
\leq \|\nabla u\|^\frac{3}{2}_{L^2}(\|G\|_{L^2}+\|\nabla u\|_{L^2}+\|P(\r)\|_{L^2}+\|P(\r)\|_{L^6})^{\frac{3}{2}}\notag\\
\leq&\delta\|\sqrt{\r}\ud\|^2_{L^2}+C\big(\br,\delta)\big(\|\nabla u\|^3_{L^2}+\|\nabla u\|^6_{L^2}+\ep^{\frac{1}{2}}\big).\label{J4}
\end{align}

Thus, substituting \eqref{J1}--\eqref{J7} into \eqref{J0}, using \eqref{J4} and choosing $\delta$ small enough, we can obtain \eqref{est1}.

Then, operating $\eta\ud^j[\frac{\pt}{\pt_t}+\div(f\cdot)]$ on $\eqref{eqb}^j_2$ and then summing with respect to $j$, we obtain, after integrating the resulting equation over $\Omega$, that
\begin{align}
&\big(\frac{\eta}{2}\int_\Omega\r|\ud|^2dx\big)_t-\frac{\eta^\prime}{2}\int_{\Omega}\r|\ud|^2dx\notag\\
=&-\eta\int_\Omega\ud^j[\pt_jP_t+\div(f\pt_jP(\r))]dx+2\mu\eta\int_\Omega\ud^j\big[\big(\div\D(u)\big)^j_t+\div\big(f(\div\D(u))^j\big)\big]dx\notag\\
&+\l\eta\int_\Omega\ud^j\big[\pt_j\pt_t(\div u)+\div(f\pt_j(\div u))\big]dx+\eta\int_{\Omega}\ud^j[(\r(\vo\times u)^j)_t+\div(f\r(\vo\times u)^j) dx\notag\\
&+\eta\int_\Omega\ud^j[(\r(u\cdot\nabla b)^j)_t+\div(f\r(u\cdot\nabla b)^j)]dx+\eta\int_\Omega\ud^j[(\r F(b)^j)_t+\div(f\r F(b)^j)]dx\notag\\
&+\mu\eta\int_\Omega\ud^j[\lap b^j_t+\div(f\lap b^j)]dx+\eta\int_\Omega\ud^j[(\r h^j)_t+\div(f\r h^j)]dx\notag\\
\triangleq&\sum^8_{i=1}M_i.\label{M0}
\end{align}
Similarly, we use the method of integrating by parts and use the equation $\eqref{eqb}_1$ and Young's inequality to estimate $M_i$ $(1\leq i\leq8)$ as follows:
\begin{align}
M_1=&-\eta\int_\Omega\ud^j[\pt_jP_t+\div(f\pt_jP(\r))]dx=\eta\int_\Omega P^\prime(\r)\r_t\pt_j\ud^j+f^k\pt_k\ud^j\pt_jPdx\notag\\
=&\eta\int_\Omega-\pt_j\ud^j P^\prime(\r)(\r\div u+f^k\pt_k\r)-P(\r)\pt_j(f^k\pt_k\ud^j)dx\notag\\
=&\eta\int_\Omega-P^\prime(\r)\r\pt_j\ud^j\div u+P(\r)\pt_ku^k\pt_j\ud^j-P(\r)\pt_jf^k\pt_k\ud^jdx\notag\\
\leq& \delta\eta\|\nabla\ud \|^2_{L^2}+C\frac{1}{\delta}\eta\|\nabla u\|^2_{L^2}+C\frac{1}{\delta}\eta m_0|\o|,\label{M1}
\end{align}
for any $\delta>0$, and
\begin{align}
M_2=&2\mu\eta\int_\Omega\ud^j\big[\big(\div\D(u)\big)^j_t+\div\big(f(\div\D(u))^j\big)\big]dx\notag\\
=&\mu\eta\int_{\pt\Omega}\ud^j(\pt_iu^j_t+\pt_ju^i_t)\n^id\Gamma-\mu\eta\int_\Omega\pt_i\ud^j(\pt_iu^j_t+\pt_ju^i_t)dx\notag\\
&-\mu\eta\int_\Omega\pt_k\ud^j\big(f^k(\pt_i\pt_iu^j+\pt_j\pt_iu^i)\big)dx\notag\\
=&-\mu\eta\sum^n_{i=1}\int_{\pt\Omega\cap Q_i}\a\psi_i\ud\cdot\big[\ud-h-\big((u\cdot\t_i)\pt_{\tau_i}u+(u\cdot\io_i)\pt_{\iota_i}u\big)\big]d\Gamma\notag\\
&-\mu\eta\int_\Omega\pt_i\ud^j(\pt_i\ud^j+\pt_j\ud^i)dx+\mu\eta\int_\Omega\pt_i\ud^j\big(\pt_i(f^k\pt_ku^j)+\pt_j(f^k\pt_ku^i)\big)dx\notag\\
&-\mu\eta\int_\Omega\pt_k\ud^jf^k(\pt_i\pt_iu^j+\pt_j\pt_iu^i)dx+\mu\eta\int_\Omega\pt_i\ud^j\big(\pt_ih^j+\pt_jh^i\big)dx\notag\\
\triangleq& \sum^5_{i=1}M_{2i}.\label{M20}
\end{align}
where we have used $\eqref{eqb}_3$, \eqref{ntau} and \eqref{udot}.

Invoking the Fourier transform and the trace theorem, we have
\begin{align}
M_{21}=&-\mu\eta\int_{\pt\Omega}\a\ud\cdot\big[\ud-h-\sum^n_{i=1}\psi_i\big((u\cdot\t_i)\pt_{\tau_i}u+(u\cdot\io_i)\pt_{\iota_i}u\big)\big]d\Gamma\notag\\
\leq&-\mu\eta\int_{\pt\Omega}\a|\ud|^2d\Gamma+C\eta\|\ud\|_{H^{\frac{1}{2}}(\pt\Omega)}\||u|^2\|_{H^{\frac{1}{2}}(\pt\Omega)}\notag\\
\leq&-\mu\eta\int_{\pt\Omega}\a|\ud|^2d\Gamma+C\eta\|\ud\|_{H^1(\Omega_5)}\||u|^2\|_{H^1(\Omega_5)}\notag\\
\leq&-\mu\eta\int_{\pt\Omega}\a|\ud|^2d\Gamma+\delta\eta\|\nabla\ud\|^2_{L^2}+C\frac{1}{\delta}\eta(\|\nabla u\|^4_{L^2}+\|\nabla u\|^4_{L^4})\label{M21}
\end{align}
for any $\delta>0$, where we have used \eqref{udot}, Sobolev's inequality and Young's inequality.

Integration by parts leads to
\begin{align}
M_{22}=&\mu\eta\int_\Omega\pt_i\ud^j\big(\pt_i(f^k\pt_ku^j)+\pt_j(f^k\pt_ku^i)\big)dx\notag\\
=&\mu\eta\int_\Omega\pt_i\ud^j\big(\pt_if^k\pt_ku^j+\pt_jf^k\pt_ku^i\big)dx+\mu\eta\int_\Omega f^k\pt_i\ud^j\big(\pt_i\pt_ku^j+\pt_j\pt_ku^i\big)dx\notag\\
=&\mu\eta\int_\Omega\pt_i\ud^j\big(\pt_if^k\pt_ku^j+\pt_jf^k\pt_ku^i\big)dx-\mu\eta\int_\Omega\pt_k\pt_i\ud^jf^k(\pt_iu^j+\pt_ju^i)dx\notag\\
&-\mu\eta\int_\Omega\pt_i\ud^j\pt_kf^k(\pt_iu^j+\pt_ju^i)dx\notag\\
=&\mu\eta\int_\Omega\pt_i\ud^j\big(\pt_if^k\pt_ku^j+\pt_jf^k\pt_ku^i\big)dx-\mu\eta\int_{\pt_\Omega}f^k\pt_k\ud^j(\pt_iu^j+\pt_ju^i)\n^id\Gamma\notag\\
&+\mu\eta\int_\Omega\pt_k\ud^j\pt_if^k(\pt_iu^j+\pt_ju^i)dx+\mu\eta\int_\Omega\pt_k\ud^jf^k(\pt_i\pt_iu^j+\pt_i\pt_ju^i)dx\notag\\
&-\mu\eta\int_\Omega\pt_i\ud^j\pt_kf^k(\pt_iu^j+\pt_ju^i)dx\notag.\label{m22}\\
\end{align}
The second term on the right side of \eqref{m22} can be estimated as follows:
\begin{align}
&-\mu\eta\int_{\pt\Omega}f^k\pt_k\ud^j(\pt_iu^j+\pt_ju^i)\n^id\Gamma\notag\\
=&-\mu\eta\sum^n_{k=1}\int_{\pt\Omega\cap Q_k}\psi_k(u\cdot\t_k)\pt_{\tau_k}
\big((\ud\cdot\t_k)\t_k^j+(\ud\cdot\io_k)\io_k^j+(\ud\cdot\n)\n^j\big)(\pt_iu^j+\pt_ju^i)\n^id\Gamma\notag\\
&+\psi_k(u\cdot\io_k)\pt_{\iota_k}
\big((\ud\cdot\t_k)\t_k^j+(\ud\cdot\io_k)\io_k^j+(\ud\cdot\n)\n^j\big)(\pt_iu^j+\pt_ju^i)\n^id\Gamma\notag\\
=&-\mu\eta\sum^n_{k=1}\int_{\pt\Omega\cap Q_k}\psi_k\big((u\cdot\t_k)\pt_{\tau_k}(\ud\cdot\t_k)\t_k^j+(u\cdot\io_k)\pt_{\iota_k}(\ud\cdot\t_k)\t_k^j\big)
\big(\pt_nu^j-u^i\pt_j\tn^i\big)\notag\\
&+\psi_k\big((u\cdot\t_k)\pt_{\tau_k}(\ud\cdot\io_k)\io_k^j+(u\cdot\io_k)\pt_{\iota_k}(\ud\cdot\io_k)\io_k^j\big)
\big(\pt_nu^j-u^i\pt_j\tn^i\big)\notag\\
&+\psi_k(u\cdot\t_k)\big[(\ud\cdot\t_k)\pt_{\tau_k}\t_k^j+(\ud\cdot\io_k)\pt_{\tau_k}\io_k^j\big](\pt_iu^j+\pt_ju^i)\n^i\notag\\
&+\psi_k(u\cdot\io_k)\big[(\ud\cdot\t_k)\pt_{\iota_k}\t_k^j+(\ud\cdot\io_k)\pt_{\iota_k}\io_k^j\big](\pt_iu^j+\pt_ju^i)\n^id\Gamma
\end{align}
where we have used \eqref{ntau}, \eqref{udot}.
Noting that
\begin{align}
\pt_{\tau_k}\t_k^j=&(\pt_{\tau_k}\t_k)^j=(\pt_{\tau_k}\t_k\cdot\t_k)\t_k^j+(\pt_{\tau_k}\t_k\cdot\io_k)\io_k^j+(\pt_{\tau_k}\t_k\cdot\n)\n^j\notag\\
=&(\pt_{\tau_k}\t_k\cdot\io_k)\io_k^j-(\t_k\cdot\pt_{\tau_k}\n)\n^j,
\end{align}
we can apply the same method as in deducing \eqref{M21} to obtain that
\begin{align}
&-\mu\eta\int_{\pt\Omega}f^k\pt_k\ud^j(\pt_iu^j+\pt_ju^i)\n^id\Gamma\notag\\
\leq&\mu\eta\sum^n_{k=1}\int_{\pt\Omega}\psi_k\big[(u\cdot\nabla\n)\cdot\ud\big]\n^j(\pt_iu^j+\pt_ju^i)\n^id\Gamma
+C\eta\|\ud\|_{H^{\frac{1}{2}}(\pt\Omega)}\||u|^2\|_{H^{\frac{1}{2}}(\pt\Omega)}\notag\\
\leq&-2\mu\eta\int_{\pt\Omega}(\pt_nh_2\cdot\ud)d\Gamma
+\delta\eta\|\nabla\ud\|^2_{L^2}+C\frac{1}{\delta}\eta(\|\nabla u\|^4_{L^2}+\|\nabla u\|^4_{L^4})\label{m222}
\end{align}
for any $\delta>0$. Therefore, \eqref{m22} and \eqref{m222} lead to
\begin{align}
M_{22}+M_{23}\leq&-\mu\eta\int_{\pt_\Omega}(\ud\cdot\t)(\pt_nh_2\cdot\t)d\Gamma+\delta\eta\|\nabla\ud\|^2_{L^2}\notag\\
&+C\frac{1}{\delta}\eta(\|\nabla u\|^2_{L^2}+\|\nabla u\|^4_{L^2}+\|\nabla u\|^4_{L^4}).\label{M23}
\end{align}
In addition, we have
\begin{align}
M_{24}=&-\mu\eta\int_\Omega\pt_i\ud^j(\pt_i\ud^j+\pt_j\ud^i)dx=-2\mu\eta\|\D(\ud)\|^2_{L^2},\label{M24}\\
M_{25}=&\mu\eta\int_\Omega\pt_i\ud^j\big(\pt_ih^j+\pt_jh^i\big)dx\leq\delta\eta\|\nabla\ud\|^2_{L^2}+C\frac{1}{\delta}\eta(\|\nabla u\|^4_{L^2}+\|\nabla u\|^4_{L^4})\label{M25}
\end{align}
for any $\delta>0$.

Substituting \eqref{M21}, \eqref{M23}, \eqref{M24} and \eqref{M25} into \eqref{M20}, we have
\begin{align}
M_2\leq& \mu\eta\int_{\pt\Omega}(\ud\cdot\t)(\pt_nh_3\cdot\t)d\Gamma-2\mu\eta\|\D(\ud)\|^2_{L^2}-2\mu\eta\int_\Omega\ud\cdot\div\D(h)dx\notag\\
&+\delta\eta\|\nabla\ud\|^2_{L^2}+C\frac{1}{\delta}\eta(\|\nabla u\|^2_{L^2}+\|\nabla u\|^4_{L^2}+\|\nabla u\|^4_{L^4})\label{M2}
\end{align}
for any $\delta>0$.

Similarly, we have
\begin{align}
M_3=&\l\eta\int_\Omega\ud^j\big[\pt_j\pt_t(\div u)+\div(f\pt_j(\div u))\big]dx\notag\\
=&\l\eta\int_\Omega\ud^j\big[\pt_j\pt_i(\ud^i-f^k\pt_ku^i-h^i)\big]dx-\l\eta\int_\Omega\pt_k\ud^jf^k\pt_j\pt_iu^idx\notag\\
=&-\l\eta\int_\Omega(\pt_j\ud^j)^2dx+\l\eta\int_\Omega\pt_j\ud^j\pt_if^k\pt_ku^idx+\l\eta\int_\Omega\pt_j\ud^j\pt_ih^idx\notag\\
&-\l\eta\int_\Omega\pt_kf^k\ud^j\pt_j\pt_iu^idx-\l\eta\int_\Omega \pt_j\ud^j\pt_ku^k\pt_iu^jdx\notag\\
&-\l\eta\int_{\pt\Omega}f^k\pt_k\ud^j\n^j\pt_iu^id\Gamma+\l\eta\int_\Omega \pt_k\ud^j\pt_jf^k\pt_iu^jdx\notag\\
\leq&-\l\eta\|\div\ud\|^2_{L^2}-\mu\eta\int_{\pt\Omega}(\ud\cdot\t)(\pt_nh_2\cdot\t)d\Gamma+\delta\eta\|\nabla\ud\|^2_{L^2}\notag\\
&+C\frac{1}{\delta}\eta(\|\nabla u\|^2_{L^2}+\|\nabla u\|^4_{L^2}+\|\nabla u\|^4_{L^4})\label{M3}
\end{align}
for any $\delta>0$, where we have used the fact from \eqref{udot} and \eqref{divu} that
\begin{align}
&-\l\eta\int_{\pt\Omega}f^k\pt_k\ud^j\pt_iu^i\n^jd\Gamma\notag\\
=&-\l\eta\sum^n_{i=1}\int_{\pt\Omega\cap Q_i}\psi_i\big[(u\cdot\t_i)\big(\pt_{\tau_i}(\ud\cdot\n)-(\ud\cdot\pt_{\tau_i}\n)\big)
+(u\cdot\io_i)\big(\pt_{\iota_i}(\ud\cdot\n)-(\ud\cdot\pt_{\iota_i}\n)\big)\big]\div ud\Gamma\notag\\
=&-\l\eta\sum^n_{i=1}\int_{\pt\Omega\cap Q_i}\big((u\cdot\nabla\tn)\cdot\ud\big)
\psi_i\big[\pt_{\tau_i}(u\cdot\t_i)+\pt_{\iota_i}(u\cdot\io_i)+\pt_n(u\cdot\tn)\big]d\Gamma\notag\\
\leq&-\l\eta\int_{\pt\Omega}(\ud\cdot\t)( \pt_nh_2\cdot\t)d\Gamma+\delta\eta\|\nabla\ud\|^2_{L^2}
+C\frac{1}{\delta}\eta\|\nabla u\|^4_{L^2}+C\frac{1}{\delta}\eta\|\nabla u\|^4_{L^4}.
\end{align}

Moreover, we have
\begin{align}
M_4=&\eta\int_{\Omega}\ud^j[\r_t(\vo\times u)^j+\r(\frac{d\vo}{dt}\times u)^j+\r\big(\vo\times (\ud-f\cdot\nabla u-h)\big)^j]-\pt_k\ud^jf^k\r(\vo\times u)^jdx\notag\\
=&\eta\int_{\Omega}\ud^j\big[-\div(\r f)(\vo\times u)^j+\r(\frac{d\vo}{dt}\times u)^j-\r\big(\vo\times (f\cdot\nabla u+h)\big)^j\big]-\pt_k\ud^jf^k\r(\vo\times u)^jdx\notag\\
=&\eta\int_{\Omega}\ud^j\big[\r f^k(\vo\times \pt_ku)^j+\r(\frac{d\vo}{dt}\times u)^j-\r\big(\vo\times (f^k\pt_k u+h)\big)^j\big]dx\notag\\
=&\eta\int_{\Omega}\r\ud^j(\frac{d\vo}{dt}\times u)^jdx-\eta\int_{\Omega}\r\ud\cdot(\vo\times h)dx\notag\\
\leq&C\eta(|\o|+|\frac{d\o}{dt}|)\|\sqrt{\r}\ud\|^2_{L^2}+C(\br)\eta(\|\sqrt{\r}u\|^2_{L^2}+\|\nabla u\|^4_{L^2}+\|\nabla u\|^4_{L^4}),\label{M4}
\end{align}
and
\begin{align}
M_5=&\eta\int_\Omega\ud^j(\r_tu^i\pt_i b^j+\r u^i_t\pt_i b^j+\r u^i\pt_i b_t^j)-\r f^k u^i\pt_i b^j\pt_k\ud^jdx\notag\\
=&\eta\int_\Omega \r f^k\pt_k(\ud^ju^i\pt_i b^j)+\r\ud^j (\ud^i-f^k\pt_ku^i-h^i)\pt_i b^j+\r\ud^j u^i\pt_i b_t^j-\r f^k u^i\pt_i b^j\pt_k\ud^jdx\notag\\
=&\eta\int_\Omega \r f^k\ud^ju^i\pt_k\pt_i b^j+\r\ud^j\ud^i\pt_i b^j+\r\ud^j u^i\pt_i b_t^jdx-\r\ud^j\pt_i b^jh^idx\notag\\
\leq& C\eta \|\nabla b\|_{L^\iy}\|\sqrt{\r}\ud\|^2_{L^2}+C\eta(\|b\|_{L^\iy}\|\nabla^2 b\|_{L^\iy}
+\||\cdot|\nabla^2 b\|_{L^\iy}|\o|)\|\sqrt{\r}\ud\|_{L^2}\|\sqrt{\r}u\|_{L^2}\notag\\
&+C(\br)\eta\|\nabla b_t\|_{L^\iy}\|\sqrt{\r}\ud\|_{L^2}\|\sqrt{\r}u\|_{L^2}+C(\br)\eta\|\nabla^2 b\|_{L^\iy}\|\sqrt{\r}\ud\|_{L^2}\|u\|^2_{L^4(\mathrm{supp}\tn)}\notag\\
\leq&C\eta (|\o|+|\frac{d\o}{dt}|)\|\sqrt{\r}\ud\|^2_{L^2}+C\eta(\|\sqrt{\r}u\|^2_{L^2}+\|\nabla u\|^4_{L^2}+\|\nabla u\|^4_{L^4}),\label{M5}
\end{align}
and
\begin{align}
M_6=&\eta\int_\Omega\ud^j[\r_t F(b)^j+\r F(b)_t^j]-f^k\pt_k\ud^j\r F(b)^jdx\notag\\
=&\eta\int_\Omega\r f^k\ud^j\pt_kF(b)^j+\r \ud^jF(b)_t^jdx\notag\\
\leq& C\eta\|\sqrt{\r}\ud\|_{L^2}\|\sqrt{\r}u\|_{L^2}\|\nabla F(b)\|_{L^\iy}+C\eta\|\r^{\frac{1}{2}}\|_{L^2}\|\sqrt{\r}\ud\|_{L^2}\|b\|_{L^\iy}\|\nabla F(b)\|_{L^\iy}\notag\\
&+C\eta\|\r^{\frac{1}{2}}\|_{L^2}\|\sqrt{\r}\ud\|_{L^2}\||\cdot|\nabla F(b)\|_{L^\iy}+C\eta\|\r^{\frac{1}{2}}\|_{L^2}\|\sqrt{\r}\ud\|_{L^2}\|F(b)_t\|_{L^\iy}\notag\\
\leq&C\eta(\|\sqrt{\r}\ud\|^2_{L^2}+ m_0\ep^{-2})(|\o|+|\frac{d\o}{dt}|+|\frac{d^2\o}{dt^2}|)+C\eta\|\sqrt{\r}u\|^2_{L^2},\label{M6}
\end{align}
and
\begin{align}
M_7=&\mu\eta\int_\Omega\ud^j(\lap b^j_t+\pt_ku^k\lap b^j+f^k\pt_k\lap b^j)dx\notag\\
\leq&C\mu\eta\|\ud\|_{L^6}
\big(\|\lap b_t\|_{L^{\frac{6}{5}}}+\|\nabla u\|_{L^2}\|\lap b\|_{L^3}+\|u\|_{L^6}\|\nabla^3 b\|_{L^{\frac{3}{2}}}+\|(|b|+|\cdot|)\nabla^3 b\|_{L^{\frac{6}{5}}}\big)\notag\\
\leq&\delta\eta\|\nabla \ud\|^2_{L^2}+C\frac{1}{\delta}\eta\ep^{\frac{4}{3}}\|\nabla u\|^2_{L^2}+C\frac{1}{\delta}\eta\ep^{\frac{1}{3}}(|\o|+|\frac{d\o}{dt}|),\label{M7}
\end{align}
for any $\delta>0$, where we have used $\eqref{eqb}_1$, \eqref{be0}--\eqref{be2}, \eqref{wiy} and H\"{o}lder's inquality, Sobolev's inequality and Young's inequality.

We use the equation \eqref{eqh} to estimate $M_8$ as follows:
\begin{align}
M_8=&\eta\int_\Omega\ud^j(\r h^j_t+\r f^k\pt_k h^j)+[\r_t+\div(\r f)]\ud\cdot hdx\notag\\
=&(2\mu+\l)\eta\int_\Omega\ud^j\lap h^jdx+\sum^4_{i=1}\eta\int_\Omega\ud\cdot I_idx.\label{M80}
\end{align}
We obtain, after using the method of integration by parts, that
\begin{align}
&(2\mu+\l)\eta\int_\Omega\ud^j\lap h^jdx\notag\\
=&(2\mu+\l)\eta\int_{\pt\Omega}\ud\cdot\pt_n(h_1+h_2) d\Gamma-(2\mu+\l)\eta\int_\Omega\pt_i\ud^j\pt_i h^jdx\notag\\
\leq &(2\mu+\l)\eta\int_{\pt\Omega}\ud\cdot\pt_nh_2 d\Gamma
+\delta\eta\|\nabla \ud\|^2_{L^2}+C\frac{1}{\delta}\eta(\|\nabla u\|^4_{L^2}+\|\nabla u\|^4_{L^4})\label{M81}
\end{align}
for any $\delta>0$, where we have used the fact that
\begin{align}
\ud\cdot\pt_n h_1=0\,\,\,\,\,\mathrm{on}\,\,\Omega.
\end{align}

Due to \eqref{rotcal}, we can use the method of integration by parts and the trace theorem to estimate the second term on the right hand side of \eqref{M80} such that the integral on the boundary contains only lower order terms. Similar to \eqref{M21}, we have
\begin{align}
\sum^4_{i=1}\eta\int_\Omega\ud\cdot I_idx\leq \delta \|\nabla\ud\|^2_{L^2}+C\frac{1}{\delta}\eta(\|\nabla u\|^2_{L^2}+\|\nabla u\|^4_{L^2}+\|\nabla u\|^4_{L^4})\label{M82}
\end{align}
for any $\delta>0$, where we have used H\"{o}lder's inquality, Sobolev's inequality and Young's inequality.

Substituting \eqref{M81} and \eqref{M82} into \eqref{M80}, we have
\begin{align}
M_8\leq &(2\mu+\l)\eta\int_{\pt\Omega}(\ud\cdot\t)\pt_nh_2\cdot\t d\Gamma
+\delta \|\nabla\ud\|^2_{L^2}+C\frac{1}{\delta}\eta(\|\nabla u\|^2_{L^2}+\|\nabla u\|^4_{L^2}+\|\nabla u\|^4_{L^4})\label{M8}
\end{align}
for any $\delta>0$.

Moreover, similarly to \eqref{J4}, it holds that
\begin{align}
\|\nabla u\|^4_{L^4}\leq& \|\nabla u\|_{L^2}\|\nabla u\|^3_{L^6}\leq\|\nabla u\|_{L^2}(\|G\|_{L^2}+\|\nabla u\|_{L^2}+\|P(\r)\|_{L^2}+\|P(\r)\|_{L^6})^3\notag\\
\leq &C\big(\|\nabla u\|_{L^2}\|\r\ud\|^3_{L^2}+\|\nabla u\|^2_{L^2}+\|\nabla u\|^4_{L^2}+\ep\big).\label{M9}
\end{align}

Note that
\begin{align}
\min\big\{2\mu,2\mu+3\l\big\}\|\D(\ud)\|_{L^2}\leq 2\mu\|\D(\ud)\|_{L^2}+\l\|\div \ud\|^2_{L^2}.\label{noteud}
\end{align}

Substituting \eqref{M1}, \eqref{M2}, \eqref{M3}, \eqref{M4}--\eqref{M7} and \eqref{M8} into \eqref{M0}, using \eqref{energy}, \eqref{M9}, \eqref{noteud} and the Lemma \ref{korn}, and choosing $\delta$ suitably small, we can obtain \eqref{est2}.

The proof is completed.
\end{proof}

We now derive the estimate on $A_2(\sg(T))$ as follows:
\begin{lemma} \label{EST3}
Let $(\r,u)$ be a smooth solution to \eqref{eqb} on $\Omega\times(0,T]$ satisfying \eqref{condition}. Suppose that $\o\in W^{4,1}(0,\iy)$, and suppose that \eqref{datab1}--\eqref{datab3} hold. Then, there exist positive constants $K$ and $\e_1$ both depending only on $\mu$, $\l$, $a$, $\g$, $\w41$, $\br$ and $M$ such that if $m_0\leq \ep^3$ and if $\ep\leq \e_1$, then
\begin{align}
A_2(\sg(T))+\int_0^{\sg(T)}\int_{\Omega}\r|\ud|^2dxdt\leq 2K.\label{est3}
\end{align}
\end{lemma}
\begin{proof}
For $t\in (0,\sg(T))$, integrating \eqref{est1} over $(0,t)$ and taking $\eta=1$, we obtain, after using \eqref{condition}, \eqref{be2}, \eqref{mass}, \eqref{energy} and \eqref{fact}, H\"{o}lder's inequality, Sobolev's inequality and Young's inequality, that
\begin{align}
&\|\nabla u\|^2_{L^2}+\int_0^{t}\int_{\Omega}\r|\ud|^2dxdt\notag\\
\leq& C\int_\Omega P(\r)\div udx\big|_0^{t}+C\int_{\Omega}\lap b\cdot udx\big|^t_0+C+C\int_0^{t}(\|\nabla u\|^3_{L^2}+\|\nabla u\|^6_{L^2})dt\notag\\
\leq& C(\br)(M+1)+\frac{1}{2}A_2(\sg(T))+C\ep[A_2(\sg(T))]^2,\label{A30}
\end{align}
which gives that
\begin{align}
A_2(\sg(T))+\int_0^{\sg(T)}\int_{\Omega}\r|\ud|^2dxdt\leq K+C_1\ep[A_2(\sg(T))]^2
\end{align}
for some positive constants $K$ and $C_1$ depending only on $\mu$, $\l$, $a$, $\g$, $\w41$, $\br$ and $M$. By choosing $\e_1\triangleq (9C_1K)^{-1}$, we get \eqref{est3}.

The proof is completed.
\end{proof}

Then, we derive the upper bounds for $A_1(T)$ as follows:
\begin{lemma}\label{EST4}
Let $(\r,u)$ be a smooth solution to \eqref{eqb} on $\Omega\times(0,T]$ satisfying \eqref{condition} for $K$ as in the Lemma \ref{EST3}. Suppose that $\o\in W^{4,1}(0,\iy)$, and suppose that \eqref{datab1}--\eqref{datab3} hold. Then, there exist a positive constant $\e_2$ depending only on $\mu$, $\l$, $a$, $\g$, $\w41$, $\br$ and $M$ such that if $m_0\leq \ep^3$ and if $\ep\leq \e_2$, then
\begin{align}
A_1(T)\leq \ep^{\frac{1}{2}}.\label{est4}
\end{align}
\end{lemma}
\begin{proof}
For simplicity, we only prove the case with $T>2$. Otherwise, it can be done by choosing the step size suitably small.

For integer $i$ $(1\leq i\leq [T]-1)$ and $t\in(i-1,i+1]$, taking $\eta=\sg_i(t)\triangleq\sg(t+1-i)$ and integrating \eqref{est1} over $(i-1,t)$, we obtain, after using \eqref{condition}, \eqref{be2}, \eqref{mass}, \eqref{energy} and \eqref{fact}, H\"{o}lder's inequality, Sobolev's inequality and Young's inequality, that
\begin{align}
&\sg_i\|\nabla u\|^2_{L^2}+\int^t_{i-1}\sg_i\|\sqrt{\r}\ud\|^2_{L^2}ds\notag\\
\leq&C\sg_i\int_\Omega P(\r)\div udx+C\sg_i\int_{\Omega}\lap b\cdot udx+C\ep+C\int^t_{i-1}\sg_i(\|\nabla u\|^3_{L^2}+\|\nabla u\|^6_{L^2})ds\notag\\
\leq&C(\br)(m^{\frac{1}{2}}_0+\ep)+\frac{\sg_i}{2}\|\nabla u\|^2_{L^2}+C\big(A^2_1(T)+A^2_2(\sg(T))\big)\ep\notag\\
\leq&C(\br,K)\ep+\frac{\sg_i}{2}\|\nabla u\|^2_{L^2},
\end{align}
which leads to
\begin{align}
\sup_{0\leq t\leq 1}(\sg\|\nabla u\|^2_{L^2})\leq C(\br,K)\ep\leq \ep^\frac{1}{2}\label{A11}
\end{align}
and
\begin{align}
\sup_{i\leq t\leq i+1}(\|\nabla u\|^2_{L^2})+\int^{i+1}_{i-1}\sg_i\|\sqrt{\r}\ud\|^2_{L^2}ds\leq C(\br,K)\ep\leq \ep^\frac{1}{2},\label{A12}
\end{align}
provided that $m_0\leq \ep^3$ and $\ep\leq \e_2\triangleq \min\{\e_1,C(\br,K)^{-2}\}$. Noting that the constant $C(\br,K)$ is independent of $i$, we can obtain \eqref{est4} from \eqref{A11} and \eqref{A12}.

The proof is completed.
\end{proof}

Next, we will derive the time-independent upper bounds for density. We first establish the estimates on $\int^{t_2}_{t_1}\|\nabla\ud\|^2_{L^2} dt$ as follows:
\begin{lemma}\label{UD2}
Let $(\r,u)$ be a smooth solution to \eqref{eqb} on $\Omega\times(0,T]$ satisfying \eqref{condition} for $K$ as in the Lemma \ref{EST3}. Suppose that $\o\in W^{4,1}(0,\iy)$, and suppose that \eqref{datab1}--\eqref{datab3} hold. In addition, assume that $m_0\leq \ep^3$ and $\ep\leq \e_2$ as in the Lemma \ref{EST4}. Then, it holds that
\begin{align}
\sup_{t\in [0,T]}(\sg^2\|\sqrt{\r}\ud\|^2_{L^2})\leq C\ep^\frac{1}{2}\label{est5}
\end{align}
and
\begin{align}
\int^{t_2}_{t_1}\sg^2\|\nabla\ud\|^2_{L^2}dt\leq C\ep(t_2-t_1)+C\ep^{\frac{1}{2}}\label{est6}
\end{align}
for $0<t_1<t_2\leq T$,
\end{lemma}
\begin{proof}
For integer $i$ $(1\leq i\leq [T]-1)$ and $t\in(i-1,i+1]$, taking $\eta=\sg^2_i$ and integrating \eqref{est2} over $(i-1,t)$, by using \eqref{energy}, \eqref{est3}, \eqref{est4} and Young's inequality, we have
\begin{align}
&\sg^2_i\|\sqrt{\r}\ud\|^2_{L^2}+\int^t_{i-1}\sg^2_i\|\nabla \ud\|^2_{L^2}ds\notag\\
\leq& C\int^t_{i-1}(|\o|\sg^2_i+\sg_i\sg_i^\prime)\|\sqrt{\r}\ud\|^2_{L^2}ds+C\ep
+C\int^t_{i-1}\sg^2_i\|\nabla u\|^4_{L^2}+\sg^2_i(\|\nabla u\|_{L^2}+\ep^{\frac{3}{2}})\|\r\ud\|^3_{L^2}ds\notag\\
\leq&C\int^t_{i-1}|\o|\sg^2_i\|\sqrt{\r}\ud\|^2_{L^2}ds+C(\br)\ep^\frac{1}{2}+CA_1(T)\ep
+C(\br)\sup_{i-1\leq s\leq t}(\|\nabla u\|_{L^2}+\ep^\frac{3}{2})\int^t_{i-1}\sg^2_i\|\r\ud\|^3_{L^2}ds\notag\\
\leq&C\int^t_{i-1}|\o|\sg^2_i\|\sqrt{\r}\ud\|^2_{L^2}ds+C(\br)\ep^\frac{1}{2}
+C\sup_{i-1\leq s\leq t}(\sg_i\|\sqrt{\r}\ud\|_{L^2})\int^t_{i-1}\sg_i\|\sqrt{\r}\ud\|^2_{L^2}ds\notag\\
\leq&C\int^t_{i-1}|\o|\sg^2_i\|\sqrt{\r}\ud\|^2_{L^2}ds+C(\br)\ep^\frac{1}{2}+C\ep^{\frac{1}{2}}\sup_{i-1\leq s\leq t}(\sg_i\|\sqrt{\r}\ud\|_{L^2}),\notag\\
\leq&C\int^t_{i-1}|\o|\sg^2_i\|\sqrt{\r}\ud\|^2_{L^2}ds+C(\br)\ep^\frac{1}{2}+\frac{1}{2}\sup_{i-1\leq s\leq t}(\sg^2_i\|\sqrt{\r}\ud\|^2_{L^2}),\label{B1}
\end{align}
where we have used the fact that $\int^t_{i-1}\sg_i\|\sqrt{\r}\ud\|^2_{L^2}ds\leq \ep^{\frac{1}{2}}$ which had been shown in \eqref{A12}.

Therefore, \eqref{B1}, together with Gronwall's inequality, gives that
\begin{align}
\sup_{0\leq t\leq 1}(\sg^2\|\sqrt{\r}\ud\|^2_{L^2})\leq C\ep^{\frac{1}{2}}\label{B2}
\end{align}
and
\begin{align}
\sup_{i\leq t\leq i+1}(\|\sqrt{\r}\ud\|^2_{L^2})\leq C\ep^{\frac{1}{2}}.\label{B3}
\end{align}
Hence, \eqref{B2} and \eqref{B3} lead to \eqref{est5}, since the constant $C$ is independent of $i$.

To estimate $\int^{t_2}_{t_1}\sg^2\|\nabla\ud\|^2_{L^2}dt$, we first integrate \eqref{est1} over $(t_1,t_2)$ for all $0<t_1<t_2\leq T$ and take $\eta=\sg$ to get
\begin{align}
&\int^{t_2}_{t_1}\sg\|\sqrt{\r}\ud\|^2_{L^2}dt\notag\\
\leq&C(\ep+A_1(T))+C\ep(t_2-t_1)+C\int^{t_2}_{t_1}\sg(\|\nabla u\|^3_{L^2}+\|\nabla u\|^6_{L^2})dt\notag\\
\leq&C(K)\ep^{\frac{1}{2}}+C\ep(t_2-t_1)+C(A^2_1(T)+A^2_2(\sg(T)))\ep\notag\\
\leq&C(K)\ep^{\frac{1}{2}}+C\ep(t_2-t_1),\label{B4}
\end{align}
where we have used \eqref{be2}, \eqref{mass}, \eqref{est3} and \eqref{est4}.

Next, integrating \eqref{est2} over $(t_1,t_2)$ for all $0<t_1<t_2\leq T$ and taking $\eta=\sg^2$,  we obtain, after using \eqref{est5} and \eqref{B4}, that
\begin{align}
&\int^{t_2}_{t_1}\sg^2\|\nabla\ud\|^2_{L^2}dt\notag\\
\leq& C\ep^{\frac{1}{2}}+C\ep(t_2-t_1)+C(\br)\sup_{t_1\leq t\leq t_2}(\sg\|\sqrt{\r}\ud\|_{L^2})\int^{t_2}_{t_1}\sg\|\sqrt{\r}\ud\|^2_{L^2}dt\notag\\
\leq& C\ep^{\frac{1}{2}}+C\ep(t_2-t_1)+C\ep^{\frac{1}{4}}\int^{t_2}_{t_1}\sg\|\sqrt{\r}\ud\|^2_{L^2}dt\notag\\
\leq& C\ep^{\frac{1}{2}}+C\ep(t_2-t_1).\label{B5}
\end{align}
This finishes the proof.
\end{proof}

Before showing the upper bounds for density, we still need the following estimation.
\begin{lemma}\label{UD1}
Let $(\r,u)$ be a smooth solution to \eqref{eqb} on $\Omega\times(0,T]$ satisfying \eqref{condition} for $K$ as in the Lemma \ref{EST3}. Suppose that $\o\in W^{4,1}(0,\iy)$, and suppose that \eqref{datab1}--\eqref{datab3} hold. In addition, assume that $m_0\leq \ep^3$ and $\ep\leq \e_2$ as in the Lemma \ref{EST4}. Then, it holds that
\begin{align}
\sup_{t\in(0,\sg(T)]}(\sg\|\sqrt{\r}\ud\|^2_{L^2})+\int_0^{\sg(T)}\sg\|\nabla\ud\|^2_{L^2}dt\leq C.\label{est7}
\end{align}
\end{lemma}
\begin{proof}
For $t\in (0,\sg(T)]$, taking $\eta=\sg$ in \eqref{est2} and integrating the resulting equation over $(0,t)$, we obtain from \eqref{est3} and Young's inequality that
\begin{align}
&(\sg\|\sqrt{\r}\ud\|^2_{L^2})+\int_0^{t}\sg\|\nabla\ud\|^2_{L^2}ds\notag\\
\leq&C\int_0^{t}(\sg|\o|+\sg^{\prime})\|\sqrt{\r}\ud\|^2_{L^2}ds+C\ep+C\int_0^{t}\sg\|\nabla u\|^4_{L^2}+\sg(\|\nabla u\|_{L^2}+\ep^{\frac{3}{2}})\|\r\ud\|^3_{L^2}ds\notag\\
\leq&C\int_0^{t}\sg|\o|\|\sqrt{\r}\ud\|^2_{L^2}ds+C(K)+C(\br)\sup_{t\in(0,\sg(T)]}(\sg^\frac{1}{2}\|\sqrt{\r}\ud\|_{L^2})\int_0^{t}\|\sqrt{\r}\ud\|^2_{L^2}ds\notag\\
\leq&C\int_0^{t}\sg|\o|\|\sqrt{\r}\ud\|^2_{L^2}ds+C(\br,K)+\frac{1}{2}\sup_{t\in(0,\sg(T)]}(\sg\|\sqrt{\r}\ud\|^2_{L^2}),
\end{align}
which, together with Gronwall's inequality, implies \eqref{est7}.

The proof of is completed.
\end{proof}

We are now in a position to use the Lemma \ref{Z2000} to derive a uniform (in time) upper bound for the density.
\begin{lemma}\label{density}
Let $(\r,u)$ be a smooth solution to \eqref{eqb} on $\Omega\times(0,T]$ satisfying \eqref{condition} for $K$ as in the Lemma \ref{EST3}. Suppose that $\o\in W^{4,1}(0,\iy)$, and suppose that \eqref{datab1}--\eqref{datab3} hold. Then, there exist a positive constant $\e$ depending only on $\mu$, $\l$, $a$, $\g$, $\w41$, $\br$ and $M$ such that if $m_0\leq \ep^7$ and if $\ep\leq \e$, then
\begin{align}
\sup_{0\leq t\leq T}\|\r\|_{L^\iy}\leq \frac{7\br}{4}.\label{rhoiy}
\end{align}
\end{lemma}
\begin{proof}
In the proof we always assume that $m_0\leq \ep^7$.

We first note that it follows from \eqref{G}, \eqref{condition}, \eqref{be0}, \eqref{be1}, \eqref{be2}, \eqref{mass}, \eqref{wiy}, \eqref{estG} and Sobolev's inequality that
\begin{align}
\|G\|_{L^6}\leq&C\big(\|\r\|_{L^\iy}\|\ud\|_{L^6}+\|\r\|_{L^\iy}\|\nabla b\|_{L^\iy}\|u\|_{L^6}+\|\r\|_{L^\iy}\|\vo\times u\|_{L^6}\notag\\
&+\|\r\|_{L^6}\|F(b)\|_{L^\iy}+\|\lap b\|_{L^6}+\|u\|_{L^6}\|u\|_{L^\iy}\big)\notag\\
\leq& C(\br)\big(\|\nabla \ud\|_{L^2}+\|\nabla u\|_{L^2}+m_0^{\frac{1}{6}}\ep^{-1}+\ep^\frac{5}{6}+\|\nabla u\|_{L^2}\|u\|_{L^\iy}\big)\notag\\
\leq& C(\br)\big(\|\nabla \ud\|_{L^2}+\|\nabla u\|_{L^2}+\ep^{\frac{1}{6}}+\|u\|_{L^\iy}\big)\notag\\
\leq& C(\br,K)\big(\|\nabla \ud\|_{L^2}+\|\nabla u\|_{L^2}+\|\sqrt{\r}\ud\|_{L^2}+\|\sqrt{\r}u\|_{L^2}+\ep^{\frac{1}{6}}\big),\label{G6}
\end{align}
where we have used the estimation from \eqref{FW4}, \eqref{estG} and Sobolev's inequality that
\begin{align}
\|u\|_{L^\iy}\leq C\|u\|_{W^{1,6}}\leq C(\|\nabla u\|_{L^2}+\|\sqrt{\r}\ud\|_{L^2}+\|\sqrt{\r}u\|_{L^2}+\ep^{\frac{1}{6}}).\label{uiy}
\end{align}

Next, rewrite the continuity equation $\eqref{eq}_1$ as
\begin{align}
D_t\r=g(\r)+\bar{b}^\prime(t),
\end{align}
where
\begin{align}
D_t\r\triangleq \r_t+(u+b-\o\times x)\cdot\nabla\r,\,\,\,\,\,\,g(\r)\triangleq-\frac{a\r}{2\mu+\l}\r^\g,\,\,\,\,\,\,\bar{b}(t)\triangleq-\frac{1}{2\mu+\l}\int_0^t\r Fdt.
\end{align}

For $t\in[0,\sg(T)]$, we deduce from the Lemma \ref{gn} and the Lemma \ref{FW}, \eqref{mass}, \eqref{energy}, \eqref{est4}, \eqref{est5}, \eqref{est7}, \eqref{G6} and \eqref{estG}, H\"{o}lder's inequality and Young's inequality that for all $0\leq t_1\leq t_2\leq \sg(T)$,
\begin{align}
|\bar{b}(t_2)-\bar{b}(t_1)|\leq&C\int_0^{\sg(T)}\|\r F\|_{L^\iy}dt\leq C(\br)\int_0^{\sg(T)}\|F\|^{\frac{1}{4}}_{L^2}(\|F\|_{L^6}+\|\nabla F\|_{L^6})^{\frac{3}{4}}dt\notag\\
\leq& C(\br)\int_0^{\sg(T)}\big(\|\nabla u\|_{L^2}+\|P(\r)\|_{L^2}\big)^{\frac{1}{4}}\notag\\
&\times\big(\|G\|_{L^6}+\|G\|_{L^2}+\|\nabla u\|_{L^2}+\|P(\r)\|_{L^2}+\|P(\r)\|_{L^6}\big)^{\frac{3}{4}}dt\notag\\
\leq&C(\br)\int_0^{\sg(T)}(\sg^{-\frac{1}{2}}(\sg^{\frac{1}{2}}\|\nabla u\|_{L^2})^{\frac{1}{4}}+\sg^{-\frac{3}{8}}m^{\frac{1}{8}}_0)
\big((\sg\|\nabla \ud\|^2_{L^2})^{\frac{3}{8}}+(\sg\|\sqrt{\r}\ud\|^2_{L^2})^{\frac{3}{8}}\big)dt\notag\\
&+C(\br)\int_0^{\sg(T)}\big(\|\nabla u\|_{L^2}+\ep^{\frac{7}{2}}\big)^{\frac{1}{4}}
(\|\nabla u\|_{L^2}+\ep^{\frac{1}{6}})^{\frac{3}{4}}dt\notag\\
\leq&C(\br,M)\ep^{\frac{1}{8}}+C(\br)\ep^{\frac{1}{16}}\int_0^{\sg(T)}(\sg^{-\frac{1}{2}}+1)
\big((\sg\|\nabla \ud\|^2_{L^2})^{\frac{3}{8}}+(\sg\|\sqrt{\r}\ud\|^2_{L^2})^{\frac{3}{8}}\big)dt\notag\\
\leq&C(\br,M)\ep^{\frac{1}{8}}+C(\br)\ep^{\frac{1}{16}}\big(1+\int^1_0\sg^{-\frac{3}{4}}dt\big)^{\frac{2}{3}}\big(\int_0^{\sg(T)}\sg\|\nabla \ud\|^2_{L^2}dt\big)^{\frac{1}{3}}\notag\\
\leq&C(\br,M)\ep^{\frac{1}{16}},
\end{align}
provided that $\ep\leq \e_2$ as in the Lemma \ref{EST4}.

Therefore, for $t\in[0,\sg(T)]$, we choose $N_0$ and $N_1$ in \eqref{N} as follows:
\begin{align}
N_0=C(\br,M),\,\,\,\,\,\,\,\,\,\,\,N_1=0.
\end{align}
Then,
\begin{align}
g(\xi)=-\frac{a\xi}{2\mu+\l}\xi^\g\leq -N_1=0,\,\,\,\,\,\mathrm{for}\,\, \mathrm{all}\,\,\xi\geq 0.
\end{align}
Thus, the Lemma \ref{Z2000} yields that
\begin{align}
\sup_{t\in[0,\sg(T)]}\|\r\|_{L^\iy}\leq \br+N_0\leq \br+C(\br,M)\ep^{\frac{1}{16}}\leq \frac{3}{2}\br, \label{den1}
\end{align}
provided that
\begin{align}
\ep\leq \e_3\triangleq\min\{\e_2,(\frac{\br}{2C(\br,M)})^{16}\}.
\end{align}

On the other hand, for $t\in[\sg(T),T]$, similarly, we can deduce from the Lemma \ref{gn}, \eqref{est4}, \eqref{est5} and \eqref{est6} that for all $\sg(T)\leq t_1\leq t_2\leq T$,
\begin{align}
|\bar{b}(t_2)-\bar{b}(t_1)|\leq&C(\br)\int_{t_1}^{t_2}\| F\|_{L^\iy}dt\leq\frac{a}{2\mu+\l}(t_2-t_1)+C(\br)\int_{t_1}^{t_2}\| F\|^{\frac{8}{3}}_{L^\iy}dt\notag\\
\leq&\frac{a}{2\mu+\l}(t_2-t_1)+C(\br)\int_{t_1}^{t_2}\|F\|^{\frac{2}{3}}_{L^2}(\|F\|_{L^6}+\|\nabla F\|_{L^6})^{2}dt\notag\\
\leq&\frac{a}{2\mu+\l}(t_2-t_1)+C(\br)\int_{t_1}^{t_2}\big(\|\nabla u\|_{L^2}+\|P(\r)\|_{L^2}\big)^{\frac{2}{3}}\notag\\
&\times\big(\|G\|_{L^6}+\|G\|_{L^2}+\|\nabla u\|_{L^2}+\|P(\r)\|_{L^2}+\|P(\r)\|_{L^6}\big)^2dt\notag\\
\leq&\frac{a}{2\mu+\l}(t_2-t_1)+C(\br)\ep^{\frac{1}{2}}(t_2-t_1)+C(\br)\ep^{\frac{1}{6}}+C(\br)\ep^{\frac{1}{6}}\int_{t_1}^{t_2}\|\nabla \ud\|^{2}_{L^2}dt\notag\\
\leq&\big(\frac{a}{2\mu+\l}+C(\br)\ep^{\frac{1}{2}}\big)(t_2-t_1)+C(\br)\ep^{\frac{1}{6}}\notag\\
\leq&\frac{3}{2}\frac{a}{2\mu+\l}(t_2-t_1)+C(\br)\ep^{\frac{1}{6}},
\end{align}
provided that
\begin{align}
\ep\leq \e_4\triangleq \min\{\e_3,\,\big(\frac{a}{2(2\mu+\l)C(\br)}\big)^2\}.
\end{align}
Choosing
\begin{align}
N_0=C(\br)\ep^{\frac{1}{6}},\,\,\,\,\,\,\,\,\,\,\,N_1=\frac{3}{2}\frac{a}{2\mu+\l}
\end{align}
in \eqref{N} and setting $\bar{\xi}=1$ in \eqref{barxi}, we note that
\begin{align}
g(\xi)=-\frac{a\xi}{2\mu+\l}\xi^\g\leq -N_1=-\frac{a}{2\mu+\l},\,\,\,\,\,\mathrm{for}\,\, \mathrm{all}\,\,\xi\geq\bar{\xi}.
\end{align}
Again, it follows from the Lemma \ref{Z2000} and \eqref{den1} that
\begin{align}
\sup_{t\in[\sg(T),T]}\|\r\|_{L^\iy}\leq \max\{\frac{3}{2}\br,1\}+N_0\leq \frac{3}{2}\br+C(\br)\ep^{\frac{1}{6}}\leq \frac{7}{4}\br,
\end{align}
provided that
\begin{align}
\ep\leq \e\triangleq\min\{\e_4,\,\frac{\br}{4C(\br)}\}.\label{den2}
\end{align}
The combination of \eqref{den1} and \eqref{den2} completes the proof.
\end{proof}

\subsection{Estimates on the spatial gradient of the smooth solution}
From now on, assume that
\begin{align}
m_0\leq \ep^7\,\,\,\,\,\,\,\mathrm{and}\,\,\,\,\,\,\,\ep\leq \e
\end{align}
as in the Proposition \ref{priori3} and the constant $C$ may depend on
\begin{align}
T,\,\,\|g\|_{L^2},\,\,\|\r_0\|_{H^2\cap W^{2,6}},\,\,\|P(\r_0)\|_{H^2\cap W^{2,6}},\,\,\||\cdot|\r_0^{\frac{1}{2}}\|_{L^\iy},
\end{align}
in addition to $\mu$, $\l$, $a$, $\g$, $\w41$, $\br$, $\a$ and $M$.

\begin{lemma} \label{UD4}
Let $(\r,u)$ be a smooth solution to \eqref{eqb} on $\Omega\times(0,T]$ satisfying \eqref{condition}. Suppose that $\o\in W^{4,1}(0,\iy)$, and suppose that \eqref{datab1}--\eqref{datab3} hold. In addition, assume that $m_0\leq \ep^7$ and $\ep\leq \e$ as in the Proposition \ref{priori3}. Then, it holds that
\begin{align}
\sup_{0\leq t\leq T}\|\sqrt{\r}\ud\|^2_{L^2}+\int_0^{T}\|\nabla \ud\|^2_{L^2}dt\leq C(g)\label{est8}
\end{align}
and
\begin{align}
\|u\|_{L^\iy}\leq C(g)\label{estiy}
\end{align}
for a positive constant $C(g)$ independent of $T$.
\end{lemma}
\begin{proof}
It suffices to show
\begin{align}
\sup_{t\in(0,\sg(T)]}\|\sqrt{\r}\ud\|^2_{L^2}+\int_0^{\sg(T)}\|\nabla \ud\|^2_{L^2}dt\leq C(g).\label{sgTT}
\end{align}
From the compatibility condition \eqref{datab3}, we can define
\begin{align}
\sqrt{\r}\ud(x,t=0)=-g-\sqrt{\r_0}\vo(0)\times u_0-\sqrt{\r_0}u_0\cdot\nabla b(0)+\sqrt{\r_0}F(b(0))-\sqrt{\r_0}h(x,t=0),\label{time0}
\end{align}
which, together with \eqref{be0}--\eqref{be2}, \eqref{mass} and \eqref{wiy}, leads to
\begin{align}
\|\sqrt{\r}\ud(t=0)\|_{L^2}
\leq&C\big(\|g\|_{L^2}+\|\r^{\frac{1}{2}}_0\|_{L^3}\|u_0\|_{L^6}+\|\r_0^{\frac{1}{2}}\|_{L^3}\|u_0\|_{L^6}\|\nabla b(0)\|_{L^\iy} \notag\\
&+\|\r_0^{\frac{1}{2}}\|_{L^2}\|F(b(0))\|_{L^\iy}+\|\r_0^{\frac{1}{2}}\|_{L^6}\|u_0\|^2_{L^6}\big)\notag\\
\leq&C\|g\|_{L^2}+C(\br,M)\ep^{\frac{7}{6}}.\label{time}
\end{align}

Thus, for $t\in(0,\sg(T)]$, taking $\eta=1$ in \eqref{est2} and integrating the resulting inequality over $(0,t)$, we get from \eqref{est3}, \eqref{time0}, \eqref{time} and Young's inequality that
\begin{align}
&\|\sqrt{\r}\ud\|^2_{L^2}+\int_0^{t}\|\nabla \ud\|^2_{L^2}ds\notag\\
\leq&C\|\sqrt{\r_0}g\|^2_{L^2}+C\ep+C(\br)\int_0^{t}(\|\nabla u\|_{L^2}+m_0^{\frac{1}{2}})\|\sqrt{\r}\ud\|^3_{L^2}ds\notag\\
\leq&C(g,\br)+C(\br,M)\sup_{t\in(0,\sg(T)]}(\|\sqrt{\r}\ud\|_{L^2})\int_0^{\sg(T)}\|\sqrt{\r}\ud\|^2_{L^2}ds\notag\\
\leq&C(g,\br,M)+\frac{1}{2}\sup_{t\in(0,\sg(T)]}(\|\sqrt{\r}\ud\|^2_{L^2}),\label{sgT}
\end{align}
which implies \eqref{sgTT}.

\eqref{sgTT}, together with the Lemma \ref{UD2}, gives \eqref{est8}. Moreover, \eqref{estiy} is a direct result from \eqref{uiy}, \eqref{energy}, \eqref{est3} \eqref{est4} and \eqref{est8}.

The proof is completed.
\end{proof}

Next, we consider the decomposition of the velocity $u=v+w$, where $v$ solves the elliptic system
\begin{equation}
\left.
\begin{cases}
2\mu\div\D(v)+\l\nabla\div v=\nabla (P(\r)-P(\tr))\,\,\,\,\,\,\,\,\,\mathrm{in}\,\,\Omega,\\
v\,\,\, \mathrm{satisfies}\,\,\eqref{nab2}\,\,\,\,\,\,\,\,\,\,\,\,\,\,\mathrm{on}\,\,\pt\Omega.
\end{cases}
\right.\label{eqV}
\end{equation}
Then, it follows from the momentum equations $\eqref{eq}_2$ and \eqref{eqV} that $w$ satisfies
\begin{equation}
\left.
\begin{cases}
2\mu\div\D(w)+\l\nabla\div w=G\,\,\,\,\,\,\,\,\,\mathrm{in}\,\,\Omega,\\
w\,\,\, \mathrm{satisfies}\,\,\eqref{nab2}\,\,\,\,\,\,\,\,\,\,\,\,\,\,\mathrm{on}\,\,\pt\Omega.
\end{cases}
\right.\label{eqW}
\end{equation}

Based on the Lemma \ref{exterior}, we show some estimates on $v$ and $w$ as follows:
\begin{lemma}\label{VW}
Let $(\r,u)$ be a smooth solution to \eqref{eqb} on $\Omega\times(0,T]$. If $v$ and $w$ satisfy \eqref{eqV} and \eqref{eqW}, respectively, then for $p\in [2,6]$, the following estimates hold:
\begin{align}
\|v\|_{L^6}\leq& C\|P(\r)\|_{L^2},\label{VW0}\\
\|\nabla v\|_{L^p}\leq& C(\|P(\r)\|_{L^p}+\|P(\r)\|_{L^2}),\label{VW1}\\
\|\nabla^2 v\|_{L^p}\leq& C(\|\nabla P(\r)\|_{L^p}+\|P(\r)\|_{L^p}+\|P(\r)\|_{L^2}),\label{VW2}\\
\|\nabla w\|_{H^1}\leq& C(\|G\|_{L^2}+\|\nabla u\|_{L^2}+\|P(\r)\|_{L^2}),\label{VW3}\\
\|\nabla^2 w\|_{L^p}\leq& C(\|G\|_{L^p}+\|G\|_{L^2}+\|\nabla u\|_{L^2}+\|P(\r)\|_{L^2}).\label{VW4}
\end{align}
\end{lemma}
\begin{proof}
Multiplying \eqref{eqV} by $v$, we obtain, after using the method of integration by parts, that
\begin{align}
\|\nabla v\|_{L^2}\leq C\|P(\r)\|_{L^2},\label{VL2}
\end{align}
which, together with Sobolev's inequality \eqref{sobolev}, gives \eqref{VW0}.

Applying the Lemma \ref{sob} and \ref{exterior}, and using \eqref{VL2}, we have
\begin{align}
\|\nabla v\|_{L^p}\leq& C(\|P(\r)\|_{L^p}+\|v\|_{L^6})\leq C(\|P(\r)\|_{L^p}+\|\nabla v\|_{L^2})\notag\\
\leq &C(\|P(\r)\|_{L^p}+\|P(\r)\|_{L^2}),\\
\|\nabla^2 v\|_{L^p}\leq &C(\|\nabla P(\r)\|_{L^p}+\|\nabla v\|_{L^p})\notag\\
\leq& C(\|\nabla P(\r)\|_{L^p}+\|P(\r)\|_{L^p}+\|P(\r)\|_{L^2}),\\
\|\nabla w\|_{H^1}\leq& C(\|\nabla^2 w\|_{L^2}+\|\nabla w\|_{L^2})\leq C(\|G\|_{L^2}+\|\nabla w\|_{L^2})\notag\\
\leq& C(\|G\|_{L^2}+\|\nabla u\|_{L^2}+\|\nabla v\|_{L^2})\notag\\
\leq& C(\|G\|_{L^2}+\|\nabla u\|_{L^2}+\|P(\r)\|_{L^2}),\\
\|\nabla^2 w\|_{L^p}\leq& C(\|G\|_{L^p}+\|\nabla w\|_{L^p})\leq C(\|G\|_{L^p}+\|\nabla w\|_{H^1})\notag\\
\leq& C(\|G\|_{L^p}+\|G\|_{L^2}+\|\nabla u\|_{L^2}+\|P(\r)\|_{L^2}).
\end{align}

The proof is completed.
\end{proof}

With the Lemma \ref{UD4} at hand, we now derive the estimate on the spatial gradient for $(\r,u)$.
\begin{lemma}\label{rp}
Let $(\r,u)$ be a smooth solution to \eqref{eqb} on $\Omega\times(0,T]$ satisfying \eqref{condition}. Suppose that $\o\in W^{4,1}(0,\iy)$, and suppose that \eqref{datab1}--\eqref{datab3} hold. In addition, assume that $m_0\leq \ep^7$ and $\ep\leq \e$ as in the Proposition \ref{priori3}. Then, it holds that
\begin{align}
\sup_{0\leq t\leq T}(\|\nabla\r\|_{L^2\cap L^6}+\|\nabla^2 u\|_{L^2})+\int^T_0\|\nabla u\|_{L^\iy}dt\leq C(T).\label{est9}
\end{align}
\end{lemma}
\begin{proof}
In the following, the constant $C$ may depend on $p$.

For $2\leq p\leq 6$, applying the operator $\nabla$ to $\eqref{eqb}_1$ and dot-multiplying the resulting equation by $p|\nabla \r|^{p-2}\nabla \r$, we obtain, after using the method of integration by parts, that
\begin{align}
\frac{d}{dt}\int_{\Omega}|\nabla \r|^pdx=&(1-p)\int_{\Omega}|\nabla \r|^p\div udx-p\int_{\Omega}|\nabla \r|^{p-2}\nabla\r\cdot\nabla f\cdot\nabla \r dx\notag\\
&-p\int_{\Omega}\r|\nabla \r|^{p-2}\nabla \r\cdot\nabla(\div u)dx\notag\\
\leq &C\int_{\Omega}|\nabla \r|^p(|\nabla u|+1)dx+C\int_{\Omega}|\nabla \r|^{p-1}|\nabla^2 u|dx\notag\\
\leq &C(\|\nabla u\|_{L^\iy}+1)\|\nabla \r\|^p_{L^p}+C\|\nabla \r\|^p_{L^p}\|\nabla^2 u\|_{L^p},\label{rhop}
\end{align}
where we have used \eqref{be1} and \eqref{wiy}.

\eqref{rhop} implies
\begin{align}
\frac{d}{dt}\|\nabla \r\|_{L^p}\leq C\big((\|\nabla u\|_{L^\iy}+1)\|\nabla \r\|_{L^p}+\|\nabla^2 u\|_{L^p}\big).\label{T0}
\end{align}

For the case that $p=6$, applying the Lemma \ref{exterior}--\ref{BW} and the Lemma \ref{VW}, together with using \eqref{FW4}, \eqref{estG}, \eqref{G6}, \eqref{condition} and \eqref{energy}, we can estimate the terms on the right side of the above inequality as follows:
\begin{align}
&\|\nabla^2 u\|_{L^6}\leq C(\|\nabla P(\r)\|_{L^p}+\|G\|_{L^6}+\|\nabla u\|_{L^6})\notag\\
\leq&C(\|\nabla P(\r)\|_{L^6}+\|\nabla \ud\|_{L^2}+\|\sqrt{\r}\ud\|_{L^2}+\|\sqrt{\r}u\|_{L^2}+\|\nabla u\|_{L^2}+\ep^{\frac{1}{6}})\notag\\
\leq&C(\|\nabla\r\|_{L^p}+\|\nabla\ud\|_{L^2}+\|\nabla u\|_{L^2}+1)\label{T1}
\end{align}
and
\begin{align}
&\|\nabla u\|_{L^\iy}\leq C(\|\nabla v\|_{L^\iy}+\|\nabla w\|_{L^\iy})\notag\\
\leq &C(\|\nabla v\|_{L^\iy}+\|\nabla w\|_{L^6}+\|\nabla^2 w\|_{L^6})\notag\\
\leq &C\big(1+\|\nabla v\|_{BMO}\ln(e+\|\nabla^2 v\|_{L^6})+\|\nabla v\|_{L^6}+\|\nabla^2 w\|_{L^6}\big)\notag\\
\leq &C\big[1+(\|P(\r)\|_{L^\iy}+\|P(\r)\|_{L^2}+\|v\|_{L^\iy})\ln(e+\|\nabla P(\r)\|_{L^6}+\|P(\r)\|_{L^p}+\|P(\r)\|_{L^2})\notag\\
&+\|P(\r)\|_{L^6}+\|\nabla u\|_{L^6}+\|G\|_{L^6}+\|G\|_{L^2}+\|\nabla u\|_{L^2}+\|P(\r)\|_{L^2}\big].\notag\\
\leq&C\big(1+\ln(e+\|\nabla \r\|_{L^p})+\|\nabla\ud\|_{L^2}+\|\nabla u\|_{L^2}\big),\label{T2}
\end{align}
where we have used the fact from Sobolev's inequality and \eqref{VW0}--\eqref{VW1} that
\begin{align}
\|v\|_{L^\iy}\leq& C\|v\|_{W^{1,6}}\leq C(\|\nabla v\|_{L^6}+\|v\|_{L^6})\leq C(\|P(\r)\|_{L^6}+\|P(\r)\|_{L^2})\leq C.\label{T3}
\end{align}
Substituting \eqref{T1} and \eqref{T2} into \eqref{T0}, we have
\begin{align}
\frac{d}{dt}(\|\nabla \r\|_{L^6}+e)
\leq& C(1+\|\nabla\ud\|_{L^2}+\|\nabla u\|_{L^2})\ln(e+\|\nabla \r\|_{L^6})\|\nabla \r\|_{L^6}\notag\\
&+C(\|\nabla\ud\|_{L^2}+\|\nabla u\|_{L^2}+1)(\|\nabla \r\|_{L^6}+e).\label{T4}
\end{align}
Both sides of \eqref{T4} divided by $\|\nabla \r\|_{L^6}+e$ lead to
\begin{align}
&\frac{d}{dt}\ln(\|\nabla \r\|_{L^6}+e)\notag\\
\leq& C(1+\|\nabla\ud\|_{L^2}+\|\nabla u\|_{L^2})\ln(e+\|\nabla \r\|_{L^6})+C(\|\nabla\ud\|_{L^2}+\|\nabla u\|_{L^2}+1).\label{T5}
\end{align}
Therefore, \eqref{T5}, together with \eqref{energy}, \eqref{est8} and Gronwall's inequality, gives that
\begin{align}
\sup_{0\leq t\leq T}(\|\nabla\r\|_{L^6})\leq C(T). \label{T6}
\end{align}

As a consequence of \eqref{energy}, \eqref{est8}, \eqref{T2} and \eqref{T6}, we have
\begin{align}
\int_0^T\|\nabla u\|_{L^\iy}dt\leq C(T). \label{T7}
\end{align}

For the case that $p=2$, by virtue of \eqref{T7}, we can use a similar argument as above to obtain that
\begin{align}
\sup_{0\leq t\leq T}(\|\nabla\r\|_{L^2})\leq C(T).\label{T8}
\end{align}

Moreover, it follows from the Lemma \ref{exterior} and the Proposition \ref{priori3}, \eqref{estG}, \eqref{est8} and \eqref{T8} that
\begin{align}
\|\nabla^2 u\|_{L^2}\leq& C(\|G\|_{L^2}+\|\nabla P(\r)\|_{L^2}+\|\nabla u\|_{L^2})\leq C(T).
\end{align}

The proof is completed.
\end{proof}

To obtain the estimate of $u_t$, we derive the upper bound for $|x|\r(x,t)$ as follows:
\begin{lemma}\label{support}
Let $(\r,u)$ be a smooth solution to \eqref{eqb} on $\Omega\times(0,T]$ satisfying \eqref{condition} for $K$ as in the Lemma \ref{EST3}. Suppose that $\o\in W^{4,1}(0,\iy)$, and suppose that \eqref{datab1}--\eqref{datab3} hold. In addition, assume that $m_0\leq \ep^7$ and $\ep\leq \e$ as in the Lemma \ref{density}. Then, it holds that
\begin{align}
\sup_{0\leq t\leq T}\||\cdot|\r^{\frac{1}{2}}\|_{L^\iy}\leq C(T,g,\||\cdot|\r_0^{\frac{1}{2}}\|_{L^\iy}).\label{rhox}
\end{align}

Moreover, if $\r_0$ has a compact support such that $\mathrm{supp}\r_0\subset B_{R_0}$ for $R_0>1$, then $\rho(t)$ has a compact support for $t \in (0,T)$ and satisfies
\begin{align}
sup\{ |x|: \r(x,t)>0\}\leq R_{0}+Ct,\label{compact1}
\end{align}
where $C$ is a positive constant independent of $T$.
\end{lemma}
\begin{proof}
For $(x,t)\in\Omega\times(0,T)$, $\r(x,t)$ can be expressed by
\begin{equation}
\r(x,t)=\r_0\big(y(0;x,t)\big)\exp\big[-\int_{0}^{t}\div u(y(s;x,t),s)ds\big],\label{rhoy}
\end{equation}
where $y=y(s;x,t)$ is the solution to
\begin{equation}
\left.
\begin{cases}
\frac{\partial}{\partial s}y(s;x,t)=u\big(y(s;x,t),s\big)+b\big(y(s;x,t),s\big)-\omega\times y(s;x,t),   \,\,\,\,\, 0\leq s\leq t, \\
y(t;x,t)=x.
\end{cases}
\right.
\end{equation}
For $0\leq s\leq t$, we have
\begin{align}
&\frac{1}{2}\frac{d( |y(s;x,t)|^2)}{ds}=|\frac{d y(s;x,t)}{ds}\cdot y(s;x,t)|\notag\\
=&|u\big(y(s;x,t),s\big)\cdot y(s;x,t)+b\big(y(s;x,t),s\big)\cdot y(s;x,t)-\big(\o\times y(s;x,t)\big)\cdot y(s;x,t)|\notag\\
=&|u\big(y(s;x,t),s\big)\cdot y(s;x,t)|\notag\\
\leq& \|u(s)\|_{L^{\infty}}|y(s;x,t)|,\label{line}
\end{align}
where we have used the fact that $b\big(y(s;x,t),s\big)\cdot y(s;x,t)=0$.

It follows from \eqref{line} that
\begin{align}
\frac{d(|y(s;x,t)|)}{ds}\leq \|u(s)\|_{L^{\infty}},
\end{align}
which, together with \eqref{estiy}, leads to the fact that
\begin{equation}
|y(0;x,t)|\geq |x|-2C(g)T   \label{y0}
\end{equation}
for any $x\in \Omega$ such that $|x|\geq 1+3C(g)T$, where $C(g)$ is the same constant as in \eqref{estiy}.

Therefore, we can obtain \eqref{rhox} from \eqref{rhoy}, \eqref{rhoiy}, \eqref{T7} and \eqref{y0}. Moreover, \eqref{compact1} can be obtained in a similar way.

The proof is completed.
\end{proof}

We are now in a position to estimate the time-derivative of $u$.
\begin{lemma}\label{timederi}
Let $(\r,u)$ be a smooth solution to \eqref{eqb} on $\Omega\times(0,T]$ satisfying \eqref{condition} for $K$ as in the Lemma \ref{EST3}. Suppose that $\o\in W^{4,1}(0,\iy)$, and suppose that \eqref{datab1}--\eqref{datab3} hold. In addition, assume that $m_0\leq \ep^7$ and $\ep\leq \e$ as in the Lemma \ref{density}. Then, it holds that
\begin{align}
\sup_{0\leq t\leq T}\|\sqrt{\r}u_t\|^2_{L^2}+\int_0^T\|\nabla u_t\|^2_{L^2}dt\leq C(T).\label{test}
\end{align}
\end{lemma}
\begin{proof}
Differentiate the equation \eqref{eqb} with respect to $t$ and then dot-multiply the resulting equation with $u_t$. Then, by virtue of \eqref{rhox}, we can use the same method as in this section to derive \eqref{test} and we thus omit the details. Here, we remark that the estimate \eqref{rhox} is sufficient to control the rotating effect term which has singularity at infinity, for example,
\begin{align}
&\int_{\Omega}\r\big((\frac{d\vo}{dt}\times x)\cdot\nabla u\big)\cdot u_tdx\notag\\
\leq& C\||\cdot|\r^{\frac{1}{2}}\|_{L^\iy}\|\r^{\frac{1}{2}}\|_{L^2}\|\nabla u\|_{L^3}\|u_t\|_{L^6}\notag\\
\leq& C(T)\|\nabla u\|_{H^1}\|\nabla u_t\|_{L^2}\notag\\
\leq& \nu\|\nabla u_t\|^2_{L^2}+\frac{1}{\nu}C(T)
\end{align}
for any $\nu>0$, where we have used \eqref{mass}, \eqref{energy}, \eqref{wiy}, \eqref{est9} and \eqref{rhox}, and H\"{o}lder's inequality, Sobolev's inequality and Young's inequality.

The proof is completed.
\end{proof}

\subsection{Higher-order estimates}

For the completeness of our proof, we list below the higher-order estimates on the smooth solution $(\r,u)$, which can be derived in a similar manner as those obtained in \cite{HLX2012,LX2018,Matsumura2}.
\begin{proposition}\label{higher}
Let $(\r,u)$ be a smooth solution to \eqref{eqb} on $\Omega\times(0,T]$ satisfying \eqref{condition}. Suppose that $\o\in W^{4,1}(0,\iy)$ and suppose that \eqref{datab1}--\eqref{datab3} hold. In addition, assume that $m_0\leq \ep^7$ and $\ep\leq \e$ as in the Proposition \ref{priori3}. Then, it holds that
\begin{align}
&\sup_{0\leq t\leq T}(\|\r,P(\r)\|^2_{H^2\cap W^{2,6}}+\|\r_t,P_t\|^2_{H^1})+\int_0^T\|\r_{tt},P_{tt}\|^2_{L^2}dt\leq C(T),\label{hest1}\\
&\sup_{0\leq t\leq T}(\sg\|\nabla u_t\|^2_{L^2}+\sg\|\nabla u\|^2_{H^2})+\int_0^T(\|\nabla u\|^2_{H^2}+\sg\|\nabla u_t\|^2_{H^1})dt\leq C(T),\\
&\sup_{0\leq t\leq T}(\sg\|\nabla u_t\|^2_{H^1}+\sg\|\nabla u\|^2_{W^{2,6}})+\int_0^T\big(\sg\|\sqrt{\r}u_{tt}\|^2_{L^2}+\sg^2\|\nabla u_{tt}\|^2_{L^2}\big)dt\leq C(T).\label{hest3}
\end{align}
\end{proposition}


\subsection{Proof of the theorem \ref{main}}
Finally, we are in a position the show the main theorem. The proof is similar to that of the Theorem 1.1 in \cite{HLX2012} and we just sketch the outline of it below (see \cite{HLX2012,Duan2012,YZ2017,LX2018} for more details).

{\bf \hspace{-1em}Proof of Theorem \ref{main}.}
There exists a sequence of $\{\r_0^{R_i}\}$ with compact supports in $\overline{\Omega_{R_i}}$, respectively, such that $\r_0^{R_i}\rightarrow\r_0$ in $L^1\cap H^2\cap W^{2,6}$ with
\begin{align}
\|\r_0^{R_i}-\r_0\|_{L^1\cap H^2\cap W^{2,6}}\leq \frac{1}{R_i}\,\,\,\,\mathrm{as}\,\,\,\,R_i\,\rightarrow\,\iy, \label{ri}
\end{align}
\begin{align}
\||\cdot|(\r_0^{R_i})^{\frac{1}{2}}\|_{L^\iy}\leq \||\cdot|\r_0^{\frac{1}{2}}\|_{L^\iy}+1.\label{ri2}
\end{align}

Let $u^{R_i}_0\in D^1\cap D^2$ solve the elliptic system
\begin{equation}
\left.\begin{cases}
-\mu\lap u^{R_i}_0-(\mu+\l)\nabla\div u^{R_i}_0+\nabla P(\r^{R_i}_0)=(\r^{R_i}_0)^{\frac{1}{2}}g+\mu\lap b\,\,\,\,\mathrm{in}\,\,\Omega,\\
u^{R_i}_0\,\,\,\,\, \mathrm{satisfies}\,\,\eqref{nab2}\,\,\,\,\mathrm{on}\,\,\partial\Omega.
\end{cases}
\right.
\end{equation}
Here, the existence of $u^{R_i}_0\in D^1\cap D^2$ is guaranteed by \cite{ADN1959,ADN1964} and the Lemma \ref{sob}. It follows from the standard theory of elliptic system that $u^{R_i}_0\rightarrow U_0-b$ in $D^1\cap D^2$ with
\begin{align}
\|u^{R_i}_0-(U_0-b)\|_{D^1\cap D^2}\leq C\frac{1}{R_i}\,\,\,\,\mathrm{as}\,\,\,\,R_i\,\rightarrow\,\iy\label{ui}
\end{align}
for some positive constant $C$.

Due to the Proposition \ref{Local}, for each $R_i$, there exists a unique local classical solution $(\r^{R_i}, u^{R_i})$ to the equation \eqref{eqb} with initial data $(\r^{R_i}_0,u^{R_i}_0)$. Then, with the help of the Propositions \ref{priori3}--\ref{higher} and the Lemma \ref{rp}--\ref{timederi}, we can use the continuity argument as in \cite{HLX2012,LX2018} to extend the local classical solution $(\r^{R_i}, u^{R_i})$ to a global one which we still denote by $(\r^{R_i}, u^{R_i})$, provided that $R_i$ is suitably large and the initial mass $\int_{\Omega}\r_0dx$ is sufficiently small. Moreover, by \eqref{ri}, \eqref{ri2} and \eqref{ui}, it is easy to check that the \textit{a priori} bounds \eqref{est0}, \eqref{energy}, \eqref{est8}, \eqref{est9}, \eqref{rhox}, \eqref{test} and \eqref{hest1}--\eqref{hest3} for the solution $(\r^{R_i}, u^{R_i})$ are independent of $R_i$. Thus, there exists a sequence $(R_j)$, $R_j\rightarrow\iy$, such that $(\r^{R_j}, u^{R_j})$ converges to a limit $(\r, u)$ in the following weak sense:
\begin{align}
&u^{R_j}\rightharpoonup u\,\,\,\,\,\,(\mathrm{weakly\,*})\,\,\,\,\,\mathrm{in}\,\,\,L^\iy(0,T;D^1\cap H^2),\notag\\
&\r^{R_j}\rightharpoonup \r\,\,\,\,\,(\mathrm{weakly\,*})\,\,\,\,\,\mathrm{in}\,\,\,L^\iy(0,T;H^2\cap W^{2,6}),\notag\\
&u^{R_j}_t\rightharpoonup u_t\,\,\,\,\,(\mathrm{weakly})\,\,\,\,\,\mathrm{in}\,\,\,L^2(0,T;D^1)
\end{align}
for any $T>0$. Moreover, $(\r, u)$ also satisfy the regularity estimates \eqref{est0}, \eqref{energy}, \eqref{est8}, \eqref{est9} and \eqref{hest1}--\eqref{hest3}. Hence, we can easily show that $(\r, U)\triangleq(\r, u+b)$ is a weak solution to the original problem \eqref{eq}--\eqref{nab1} satisfying the regularity \eqref{class}. Again, the estimates \eqref{est0}, \eqref{energy}, \eqref{est8}, \eqref{est9}, \eqref{rhox}, \eqref{test} and \eqref{hest1}--\eqref{hest3} show that $(\r, U)$ is in fact the unique strong solution defined on $\Omega\times(0,T]$ for any $0<T<\iy$.

Moreover, \eqref{compact} follows from the Lemma \ref{support}. Therefore, the proof is completed.  $\hspace{2.7em}\Box$\\



\end{document}